\documentclass{article}

\usepackage{amsmath,amsfonts,amssymb,amsthm,graphicx}

\def\d#1{{#1\kern-0.4em\char"16\kern-0.1em}}
\def\D#1{{\raise0.2ex\hbox{-}\kern-0.4em #1}}

\newtheorem{theorem}{Theorem}[section]
\newtheorem{definition}{Definition}[section]
\newtheorem{proposition}{Proposition}[section]
\newtheorem{remark}{Remark}[section]
\newtheorem{lemma}{Lemma}[section]
\newtheorem{corollary}{Corollary}[section]
\theoremstyle{definition}\newtheorem{example}{Example}[section]

\def\N{{\mathcal{N}}}
\def\R{{\mathcal{R}}}
\def\U{{\mathcal{U}}}
\def\D{{\cal{D}}}

\def\S{{\mathcal{S}}}

\def\B{{\mathcal{B}}}
\def\Q{{\mathcal{Q}}}
\def\P{{\mathcal{P}}}

\def\NN{\mathbb{N}}
\def\RR{\mathbb{R}}

\newcommand{\mat}[4]{\left[\begin{array}{cc}#1 & #2 \\ #3 & #4 \\
	\end{array}\right]}

\author{Bogdan D. Djordjevi\'c\footnote{Mathematical Institute of the Serbian Academy of Sciences and Arts, Belgrade, Republic of Serbia.\qquad {\tt bogdan.djordjevic@turing.mi.sanu.ac.rs;\ bogdan.djordjevic93@gmail.com}}, Neboj\v sa \v C. Din\v ci\'c\footnote{University of Ni\v s, Faculty of Sciences and Mathematics, Ni\v s, Republic of Serbia.\qquad {\tt nebojsa.dincic@pmf.edu.rs;\ ndincic@hotmail.com}}, and Yingshan Chen\footnote{Student at University of Sorbonne, Paris, France. yingshan.chen@etu.sorbonne-universite.fr; chysh1998@163.com}}
\date{}

\begin{document}
	
	\title{All doubly stochastic solutions to
		$AXA = XAX$ when $A$ is a permutation matrix}
	
	\maketitle

	\begin{abstract}
		In this paper we obtain all doubly-stochastic solutions to the Yang-Baxter-like matrix equation $AXA=XAX$ when $A$ is a permutation matrix. We characterize the permutation solutions, proper doubly stochastic solutions, and we accompany our results with suitable examples. 
	\end{abstract}

	{\bf MSC 2020 Classification:} {15A24,	15B51, 15A10, 15A21.}
	
	{\bf Keyphrases:} Yang-Baxter-like matrix equation, permutation matrices, doubly stochastic inner inverses, doubly stochastic partial isometries, doubly stochastic idempotents.
	

	\section{Introduction}
	For a given square matrix $A\neq0$ (usually called a coefficient matrix), the Yang-Baxter-like matrix equation (YBME) is given as \begin{equation}
		\label{YBME} AXA=XAX.\end{equation}
	It is a quadratic matrix equation which  models various phenomena in mathematics, physics, computer science, and engineering, consult e. g.  \cite{RJB}, \cite{ABH}, \cite{Fel},  \cite{NVI}, \cite{WSACXX}, and \cite{NCY}. Solving YBME \eqref{YBME} implies finding all the matrices $X$ which satisfy the above identity, that is, finding the solution set
	$$\S(A)=\{X: AXA=XAX\}.$$ 
	Clearly there are always at least two solutions, $X=0$ and $X=A$, which are called \emph{trivial solutions}. For any $X\in \S(A)$ such that $AX=XA$ we say that it is a {\it commuting solution}, while the other solutions are {\it non-commuting}. Solving \eqref{YBME} completely (finding all solutions) is still an open problem in mathematics, see \cite{ChenYong}, \cite{DehShi}, \cite{NCDBDD1}, \cite{NCDBDD2},  \cite{DinRhee2015}.
	
	In this paper it is assumed that $A$ is a permutation matrix, and we proceed to characterize and find all doubly stochastic solutions to \eqref{YBME}. This setting is naturally connected with Markov processes, Fast Fourier transform, Object-centered machine learning, scattering problems in statistical mechanics, braiding properties, etc. \cite{ABH}, \cite{{AFTNRSMT}}, \cite{BDD}, \cite{BDDcor}, \cite{NVI}, \cite{CKL}, \cite{WSACXX}, and \cite{ASRKGO}. As opposed to the traditional approaches used in solving such matrix equations, here we solve the problem by utilizing operator theory tools, and prove some convenient and useful properties which help us describe the solution set entirely. In particular, we prove that all doubly stochastic solutions are normal partial isometries, which in return implies that they all have an inner inverse which is also doubly stochastic. This is a natural generalization of the permutation-solution case, where each solution (being a permutation matrix itself) is a normal doubly stochastic isometry, therefore it has a proper inverse which is also a normal doubly stochastic isometry.
	
	 Findings in this paper offer a completion and a correction of some previous papers which have also studied this problem. In particular, all solutions obtained in \cite{BDD} are permutation solutions, but they are not the only solutions. Paper \cite{DinRhe2012} uses Brouwer fixed point theorem to prove the existence of a doubly stochastic solution, however the Brouwer fixed point theorem does not have a constructive proof and fails to provide a way of finding the hypothetical solution. Finally, paper \cite{ASRKGO} offers an algorithm for obtaining some doubly stochastic solutions, but not all of them.

	\subsection{Preliminaries}

	Throughout the paper we work with real $n-$dimensional vector spaces and square $n-$dimensional matrices. The basis of $\RR^n$ is assumed to be the Euclidean standard orthonormal system $e_1,\ldots,e_n$, where $e_i=[\delta_k^i]_{k=\overline{1,n}}$ with $\delta_i^k$ being the Kronecker delta symbol. A matrix $L\in\RR^{n\times n}$ is said to be a permutation matrix if there exists a permutation $\pi_L$ of the index set $\{1,\ldots, n\}$ such that $Le_i=e_{\pi_L(i)}$, for every $i\in\{1,\ldots,n\}$. It is straightforward to see that permutation matrices are invertible, with $L^{-1}$ corresponding to the inverse permutation $\pi^{-1}_L$ of the indexes, and the product of two permutation matrices is again a permutation matrix, corresponding to the composition of the initial permutations acting on the index set. If we $S_n$ denotes the symmetric group on $\{1,\ldots,n\}$, and
	$\P_n$ denotes the group of all $n\times n$ permutation matrices, then the natural isomorphism $L\leftrightarrow \pi_L$ is obvious. In the basis $(e_1,\ldots,e_n)$, a permutation matrix $L$ has a nonzero entry on its main diagonal, if and only if the permutation $\pi_L$ has a cycle of length $1$, if and only if $\pi_L$ has a fixed point. The identity matrix $I_n$ is a permutation matrix, and it fixes every base vector $e_i$. Therefore the corresponding permutation $\pi_{I_n}$ is the identity mapping denoted as $\operatorname{id}_n$, where $\operatorname{id}_n(k)=k$, $1\leq k\leq n$. From an operator perspective, every permutation matrix is a unitary (or orthogonal) operator.

	A real square matrix is a doubly stochastic (DS or d. s. for short) matrix if its entries are nonnegative, and its every row sum and every column sum equals to one. Clearly every permutation matrix is also DS matrix, while the converse is not true. Still, the product of any two DS matrices is again a DS matrix, therefore they form a non-commutative semigroup with the unit $I_n$ (see \cite{HKF}). We recall an important characterization of permutations in terms of DS matrices: a matrix $A$ is a permutation matrix if and only if it is invertible, and both $A$ and $A^{-1}$ are doubly stochastic matrices.

	It is a known fact that any doubly stochastic matrix can be expressed as a convex combination of permutation matrices. Furthermore, (see \cite{RB}) the set
	of doubly stochastic matrices is a convex and compact polytope  in $\RR^{n\times n}$, and its vertices are the permutation matrices. This polytope is known as the Birkhoff polytope, usually denoted by $\B_n$. Due to its peculiar nature, it is naturally associated with braid groups that are spanned by its vertices, and is often used in applied statistics and computational methods due to its convexity, compactness, and connections to probability distributions.

\subsection{Partition of the solution set}	
	
	Often, we will use the similarity-invariance of the YBME \eqref{YBME}: if $T$ is an invertible matrix, then (see e. g. \cite{NCDBDD1} and \cite{NCDBDD2}):
	\begin{equation}
		\label{similarity}
		TAT^{-1}XTAT^{-1}=XTAT^{-1}X\Leftrightarrow A\left(T^{-1}XT\right)A=\left(T^{-1}XT\right)A\left(T^{-1}XT\right).\end{equation}
	Specially, if we assume that $T$ is a permutation matrix, then $T^{-1}XT$ is a DS matrix if and only if $X$ is a DS matrix, therefore solving the initial problem \eqref{YBME} in the realm of DS matrices is not lost. In this case, the permutation matrix $T$ will be interpreted as a rearrangement of the basis vectors $e_i$.

	When solving \eqref{YBME} in the domain of doubly stochastic matrices, one encounters an immediate partition of the solution set: one consisting out of permutation solutions (denoted by $\S_P(A)$), and the other consisting out of the doubly stochastic matrices which are not permutations. We refer to the latter as {\it proper doubly stochastic solutions} (pDS for short), and we denote that set as $\S_{pDS}(A)$. When suitable, we will emphasize the difference between the permutation solutions, and the  proper doubly stochastic solutions. This splitting of the solution set is achieved by the following result:
	
	\begin{lemma}\label{regsing}
		For a given permutation matrix $A$, let $X$ be a DS solution to \eqref{YBME}. The matrix $X$ is invertible if and only if it is a permutation matrix.
	\end{lemma}
	\begin{proof} The converse statement is obvious thus we prove the direct statement only. If $X$ is a regular matrix, then from \eqref{YBME} it follows that $\det A=\det X$. Meanwhile $A$ is a permutation matrix so all its eigenvalues are on the unit circle in the complex plane, ergo
		$$\begin{aligned}
			1&=\prod_{\lambda(A)\in\sigma(A)}\left|\lambda(A)\right|=\left|\prod_{\lambda(A)\in\sigma(A)}\lambda(A)\right|=\left|\det A\right|=\\
			&=\left|\det X\right|=\left|\prod_{\lambda(X)\in\sigma(X)}\lambda(X)\right|=\prod_{\lambda(X)\in\sigma(X)}\left|\lambda(X)\right|.
		\end{aligned}$$
		On the other hand, since $X$ is a regular DS matrix, it follows that $0<|\lambda(X)|\leq 1$, for every $\lambda(X)\in\sigma(X)$. Combining the two, it follows that $|\lambda(X)|=1$, for every $\lambda(X)\in\sigma(X)$. By Hadamard's inequality (rather, equality), it follows that the column vectors in $X$ are mutually orthogonal, therefore $X$ must be a permutation matrix. 
	\end{proof}
	
	Therefore, for a given $A\in \P_n$, we will find all DS solutions to the YBME equation $AXA=XAX$, that is, find the following sets
	\begin{align*}
		\S_P(A)&=\{X\in \P_n: AXA=XAX\},\\
		\S_{DS}(A)&=\{X\in \B_n: AXA=XAX\},\\
		\S_{pDS}(A)&=\S_{DS}(A)\setminus \S_P(A).
	\end{align*} 
	The reason we emphasize those sets becomes clear once we note that for a permutation matrix $A$ there may be solutions to the YBME that are neither permutation nor proper DS matrices. Indeed, for $A=I_n\in\P_n$ corresponding to $I\in S_n$, any idempotent matrix $X$ is a solution to the YBME, and idempotent matrices comprise a much wider class than the DS matrices.
	
	\subsection{Uniform distribution matrices}
	Notice that the $n\times n$ uniform distribution matrix $U_n$, defined as
	\begin{equation}
		\label{UDM}
		U_n=\frac{1}{n}\mathbf{1}_n=\left[\begin{array}{ccccc}\frac{1}{n} &\frac{1}{n} & \frac{1}{n}& \ldots & \frac{1}{n}\\
			\frac{1}{n} &\frac{1}{n} & \frac{1}{n}& \ldots & \frac{1}{n}\\
			\vdots & \vdots& \ddots & \vdots &\vdots\\
			\frac{1}{n} &\frac{1}{n} & \frac{1}{n}& \ldots & \frac{1}{n}
		\end{array}\right],\end{equation}
	has the property that $PU_n=U_nP=U_n=U_n^2$, for any permutation matrix $P$, and
	$$PU_nP=PU_n=PU_n^2=U_nPU_n,$$
	thus $U_n$ is always a solution to the Yang-Baxer-like matrix equation with a permutation coefficient matrix. In a sense, it can be referred to as a trivial DS solution, since the remaining trivial solution $X=0$ does not belong to the set of doubly stochastic matrices. This observation motivates us to investigate some natural modifications of the uniform distribution matrix $U_n$. As it turns out, these modifications precisely comprise the class of all DS solutions when the matrix $A$ is a permutation matrix, as demonstrated below.

	Let $n\in\mathbb{N}$ be fixed. For an $n_0\in\{1,\ldots,n\}$, an $n-$dimensional DS matrix is said to be a uniform distribution of order $n_0$, if its nonzero entries are all equal to $\frac{1}{n_0}$. 
	In other words, every DS matrix which is a uniform distribution of order $n_0$ has precisely $n_0$ elements in each row and in each column with the entry $\frac{1}{n_0}$, while the remaining $n-n_0$ elements in every row and column are zeros. It is straightforward to see that when $n_0=n$ then one gets the uniform distribution matrix \eqref{UDM}, while, when $n_0=1$, one gets a permutation matrix. We denote by 
	\begin{equation}
		\label{Unn0}
		\U_n(n_0)=\{U: U-\textrm{is a DS matrix which is uniform of order $n_0$}\}.
	\end{equation}
	
	\section{Permutation solutions}\label{Perm}
			
	In this section we characterize the permutation-matrix solutions of \eqref{YBME} for an arbitrary $A\in \P_n$. We emphasize that paper \cite{BDD} offers one class of permutation solutions, but the said paper states incorrectly that those are the only permutation solutions (in fact, there are so much more). For this the first author apologizes and takes full responsibility. For the sake of completeness and self-readibility, in what follows we describe the solution set $\S_P$ utterly differently compared to \cite{BDD}, and we point out that all solutions from the said paper are contained in our description as well.
	
	In this setting, it is often more natural to work with the corresponding permutations of indices rather than the matrices themselves. Thus \eqref{YBME} reads
	\begin{equation}
		\label{ybeperm} \pi_A\pi_X\pi_A=\pi_X\pi_A\pi_X.
	\end{equation}
	
Solving \eqref{ybeperm} effectively and completely is still an open problem, and is the core subject of algebraic topology and knot theory. On this occasion we characterize the solutions to \eqref{ybeperm} in terms of orbits and actions, and explain why our approach works for any permutation $\pi_A$, although the effectiveness becomes more difficult as the number of cycles increases, especially the number of cycles having the same length. Still, our findings suffice for Section \ref{cyclic} and Section \ref{general}.

	For any $\pi\in S_n$, the orbit of a point $i\in\{1,...,n\}$ is defined as the set 
	$$\{\pi^k(i): k\in\mathbb{Z}\}=\{\pi^k(i): k=0,1,...,\operatorname{ord}(\pi)-1\}.$$ 
	These orbits are precisely the sets of entries occurring in disjoint cycles of the permutation $\pi$, including singletons which correspond to fixed points. This gives rise to the following classification of permutation solutions: 
		\begin{definition}
		Let $A\in\P_n$, and let $\Omega_1,\ldots,\Omega_m$ be the orbits of the
		permutation $\pi_A\in S_n$ corresponding to $A$. A permutation solution
		$X\in\mathcal S_P(A)$ is called orbit-preserving if the corresponding
		permutation $\pi_X$ satisfies
		$$
		(\forall i\in\{1,...,m\})\;\pi_X(\Omega_i)=\Omega_i.
		$$
		Otherwise, $X$ is called intertwining.
	\end{definition}
	For example, for $\pi_A=(12)(3)=C_1C_2$, we have $\Omega_1=\{1,2\},\;\Omega_2=\{3\}$, so orbit-preserving means $\pi_X(\{1,2\})=\{1,2\}$ and $\pi_X(\{3\})=\{3\}$, which agrees with the cycles. 
	
	Let us denote by $C_{\P_n}(A)=\{B\in \P_n:AB=BA\}$ the centralizer of $A$ in $\P_n$. Also, since the cyclic group $\langle A\rangle$ acts on $\mathcal S_P(A)$ by conjugation, $$A^k\circ X:=A^kXA^{-k},$$ 
	the orbit of $X\in\mathcal S_P(A)$ under this action (so-called $A$-orbit of $X$) will be denoted by
	$$
	\mathcal O_A(X)=\{A^kXA^{-k}:k\in\mathbb Z\}.
	$$

	\begin{theorem}\label{Centralizer_decomposition}
		Let $A\in \P_n$. For each $C\in C_{\P_n}(A)$, define the $C$-sector
		$$
		\mathcal Q_C(A)=\{Q\in \P_n:Q^3=(QA)^2=C\}.
		$$
		Then the solution set $\S_P(A)$ is the disjoint union of solution sectors $A^{-1}\Q_C(A)$:
		$$
		\mathcal S_P(A)=\bigcup_{C\in C_{\P_n}(A)}A^{-1}\mathcal Q_C(A).
		$$
		
		Equivalently, $X\in\mathcal S_P(A)$ if and only if, for $Q=AX,$
		there exists a unique $C\in C_{\P_n}(A)$ such that $Q^3=(QA)^2=C.$
		In that case,
		$$
		X=A^{-1}Q=QAQ^{-1}.
		$$
		In particular, every permutation solution $X$ is permutation-similar to $A$ and hence has the same cycle type as $A$.
	\end{theorem}
	
	\begin{proof}
		Suppose first that $X\in\mathcal S_P(A)$ and set $Q=AX.$ Since $AXA=XAX,$
		we obtain $Q^3=(AX)^3=(AXA)(XAX)=(AXA)^2=(QA)^2.$ Let $C:=Q^3=(QA)^2.$
		Clearly, $C$ commutes with $Q$, since $C=Q^3$, and it also commutes with
		$QA$, since $C=(QA)^2$. Since $A=Q^{-1}(QA),$ it follows that
		$$
		AC=Q^{-1}(QA)C=Q^{-1}C(QA)=CQ^{-1}(QA)=CA.
		$$
		Therefore, $C\in C_{\P_n}(A),$ and hence $Q\in\mathcal Q_C(A).$
		Since $X=A^{-1}Q$, we conclude that $X\in A^{-1}\mathcal Q_C(A).$
		
		Conversely, let $C\in C_{\P_n}(A)$ and let $Q\in\mathcal Q_C(A).$
		Then $Q^3=(QA)^2,$ and left multiplication by $Q^{-1}$ yields $Q^2=AQA$.
		Now for $X=A^{-1}Q$ we have
		$$
		XAX=A^{-1}QAA^{-1}Q=A^{-1}Q^2=A^{-1}AQA=QA=AXA.$$
		Thus $X\in\mathcal S_P(A)$.
		
		The union is disjoint. Indeed, for a given solution $X$, the permutation
		$Q=AX$ is uniquely determined, and consequently so is $C=Q^3.$
		
		Finally, from $Q^3=(QA)^2$ we obtain $Q^2=AQA.$ Therefore, $A^{-1}Q=QAQ^{-1},$
		and thus $X=QAQ^{-1}.$ Hence $X$ is permutation-similar to $A$, and consequently $X$ and $A$ have the same cycle type.
	\end{proof}
	
	\begin{remark}
		For each $C\in C_{\P_n}(A)$, the sector $\mathcal Q_C(A)$ can be written as
		$$
		\mathcal Q_C(A)=\operatorname{Cub}(C)\cap \operatorname{Sqr}(C)A^{-1},
		$$
		where $\operatorname{Cub}(C)=\{Q\in \P_n:Q^3=C\}$ and
		$\operatorname{Sqr}(C)=\{R\in \P_n:R^2=C\}.$
		Thus characterization of a $C$-sector reduces to a simultaneous cube- and square-root problem.
	\end{remark}
	
	\begin{remark}
		Theorem \ref{Centralizer_decomposition} reduces the permutation solution problem for an arbitrary
		permutation matrix $A$ to the determining the sectors
		$\mathcal Q_C(A),\; C\in C_{\P_n}(A).$
		When $A$ is a single cycle, its centralizer is cyclic,
		$C_{\P_n}(A)=\langle A\rangle,$
		so the sectors are indexed simply by the powers $C=A^k$.
		
		For a permutation consisting of several cycles, the centralizer is
		generally more complicated, especially in the presence of cycles of
		the same length.
	\end{remark}
	
	\begin{remark}
		Actually, because of $\pi_X=(\pi_A \pi_X)\pi_A (\pi_A \pi_X)^{-1},$ it follows that $\pi_X$ and $\pi_A$ are permutation-similar in $S_n$, hence they have the same number of cycles of the same lengths.
	\end{remark}
	
	\begin{example}\label{Ex_4,2}
		Let $A\in\P_6$ be the permutation matrix associated with the permutation $\pi_A=(1\,2\,3\,4)(5\,6)\in S_6.$ Write $a=(1\,2\,3\,4),\;b=(5\,6),$
		so that $\pi_A=ab$. Since the two cycles have distinct lengths,
		$C_{S_6}(\pi_A)=\langle a\rangle\times\langle b\rangle,$
		and hence
		$$
		C_{S_6}(\pi_A)=\{I,b,a,ab,a^2,a^2b,a^3,a^3b\},
		$$
		where $a^2=(1\,3)(2\,4).$
		
		Via the natural isomorphism between $\mathcal P_6$ and $S_6$, the matrix
		sector $\mathcal Q_C(A)$ corresponds to
		$$
		\mathcal Q_{\pi_C}(\pi_A)=\{\pi_Q\in S_6:\pi_Q^3=(\pi_Q\pi_A)^2=\pi_C\}.
		$$
		We determine the nonempty sectors $\Q_{\pi_C}(\pi_A)$ by considering separately the square- and cube-root conditions corresponding to Theorem \ref{Centralizer_decomposition}.
		
		We first use the square-root condition. If $\Q_{\pi_C}(\pi_A)\neq\emptyset$, then $(\pi_Q \pi_A)^2=\pi_C$ for some $\pi_Q\in S_6$, so $\pi_C$ must be a square in $S_6$.
		Recall that squaring an odd cycle preserves its length, whereas
		squaring a cycle of an even length $2r$ splits it into two cycles of the same
		length $r$. Consequently, in the square of a permutation, the number
		of cycles of each even length is even.
		
		The relevant elements of $C_{S_6}(\pi_A)$ have the following cycle types:
		$$
		\begin{array}{c|c}
			\pi_C & \text{cycle type}\\
			\hline
			b & 2\,1^4\\
			a,\ a^3 & 4\,1^2\\
			ab,\ a^3b & 4\,2\\
			a^2 & 2^2\,1^2\\
			a^2b & 2^3,
		\end{array}
		$$
		while $I$ has cycle type $1^6$. Remark that the cycle type e.g. $2\,1^4$ means a permutation having one $2$-cycle and four $1$-cycles (i.e. fixed points). Hence $b, a, ab, a^2b, a^3, a^3b$ cannot be squares in $S_6$. Therefore, $\Q_{\pi_C}(\pi_A)=\emptyset$ for
		$\pi_C\in\{b,a,ab,a^2b,a^3,a^3b\},$ and only the sectors $\pi_C=I$ and $\pi_C=a^2$ remain.
		
		Consider first $\pi_C=a^2$. Since $\pi_Q^3=a^2=(1\,3)(2\,4),$ we must have $\pi_Q=a^2$. Indeed, the cube of an $r$-cycle consists of $\gcd(r,3)$ cycles of length $r/\gcd(r,3)$. Hence a transposition in $\pi_Q^3$ can arise only from a transposition or from a $6$-cycle. A $6$-cycle, however, produces three transpositions under cubing, whereas $a^2$ contains exactly two. Thus the transpositions $(1\,3)$ and $(2\,4)$ must themselves be cycles in $\pi_Q$, while the remaining two points must be fixed. Therefore $\pi_Q=a^2$. Moreover, $(a^2\pi_A)^2=(a^3b)^2=a^2$, and, consequently,
		$$
		\Q_{a^2}(\pi_A)=\{a^2\}.
		$$
		
		It remains to consider the sector $\pi_C=I$. Here $\pi_Q^3=I$ implies that $\pi_Q$ consists only of cycles of lengths $1$ and $3$. Thus $\pi_Q$ has one of the cycle types
		$$
		1^6,\; 3\,1^3,\; 3^2.
		$$
		A direct check among these remaining cube roots shows that the
		additional condition $(\pi_Q \pi_A)^2=I$ retains precisely the following eight permutations:
		\begin{align*}
			\Q_{I}(\pi_A)=\{&
			(2\,4\,3),\,
			(1\,3\,2),\,
			(1\,4\,2),\,
			(1\,4\,3),\,
			(1\,4\,6)(2\,5\,3),\\
			&(1\,4\,5)(2\,6\,3),\,
			(1\,5\,2)(3\,6\,4),\,
			(1\,6\,2)(3\,5\,4)
			\}.
		\end{align*}
		
		Thus the only nonempty sectors are $\Q_{I}(\pi_A)$ and $\Q_{a^2}(\pi_A).$
		By Theorem \ref{Centralizer_decomposition}, every permutation solution
		is obtained from $\pi_X=\pi_A^{-1}\pi_Q.$
		The permutations corresponding to all permutation-matrix solutions are therefore
		\begin{align*}
			\{&
			(1\,2\,3\,4)(5\,6),\,
			(1\,4\,2\,3)(5\,6),\,
			(1\,2\,4\,3)(5\,6),\,
			(1\,3\,2\,4)(5\,6),\,
			(1\,3\,4\,2)(5\,6),\\
			&(1\,3)(2\,6\,4\,5),\,
			(1\,3)(2\,5\,4\,6),\,
			(1\,6\,3\,5)(2\,4),\,
			(1\,5\,3\,6)(2\,4)
			\}.
		\end{align*}
		Hence there are precisely nine permutation solutions.
		
		The first five solutions preserve the two orbits, $\{1,2,3,4\}$ and $\{5,6\}$, of $\pi_A$, whereas the remaining four intertwine them. In particular, even for a coefficient permutation consisting of only two cycles, the sector $\pi_C=I$ may contain both orbit-preserving and intertwining solutions. \hfill$\clubsuit$
	\end{example}
	
	\begin{theorem}
		Let $\pi_A\in S_n$ be a permutation consisting out of $m$ cycles, $\pi_A=C_1C_2\dots C_m$, where $|C_i|=r_i$, $1\leq i\leq m$, and let
		$$
		L=\operatorname{lcm}(r_1,r_2,\ldots,r_m)
		=\operatorname{ord}(A).
		$$
		If $X\in\mathcal S_P(A)$, then its $A$-orbit $\mathcal O_A(X)$ is finite and
		$|\mathcal O_A(X)|\mid L.$
	\end{theorem}
	
	\begin{proof}
		Consider the map
		$\varphi:\mathcal S_P(A)\longrightarrow\mathcal S_P(A),\;\varphi(X)=AXA^{-1}.$
		Since conjugation by $A$ preserves the equation $AXA=XAX$, the map
		$\varphi$ is well defined on $\mathcal S_P(A)$.
		
		Since $A^L=I,$ we have $\varphi^L(X)=A^LXA^{-L}=X$ for every $X\in\mathcal S_P(A)$.
		
		Let $d$ be the cardinality of the orbit $\mathcal O_A(X)$, equivalently,
		the smallest positive integer such that $\varphi^d(X)=X,$ or $A^dXA^{-d}=X.$
		Write $L=qd+r,\; 0\leq r<d.$ Then
		$X=\varphi^L(X)=\varphi^{qd+r}(X)=\varphi^r(X).$
		By the minimality of $d$, it follows that $r=0$. Hence $d\mid L,$
		which proves the assertion.
	\end{proof}
	
	One may use the equivalent formula
	$|\mathcal O_A(X)|=\min\{d\geq 1: A^d X=XA^d\}.$
	
	\begin{corollary}\label{cor_orbits}
		Let $A\in \P_n$ and let $X\in\mathcal S_P(A)$. Put
		$$
		Q=AX,
		\qquad
		Q^3=(QA)^2=C\in C_{\P_n}(A).
		$$
		Then the $A$-orbit of $X$ is contained in the corresponding solution
		sector:
		$$
		\mathcal O_A(X)
		\subseteq
		A^{-1}\mathcal Q_C(A).
		$$
		Equivalently, if $Q_k=A^kQA^{-k},\; k\in\mathbb Z,$ then $Q_k\in\mathcal Q_C(A)$ for every $k\in\mathbb Z$.
	\end{corollary}
	
	\begin{proof}
		For $Q_k=A^kQA^{-k},$ we have, since $C\in C_{\P_n}(A)$,
		$Q_k^3=A^kQ^3A^{-k}=A^kCA^{-k}=C.$
		Moreover, $Q_kA=A^kQAA^{-k}$, and hence
		$$
		(Q_kA)^2=A^k(QA)^2A^{-k}=A^kCA^{-k}=C.
		$$
		Thus $Q_k\in\mathcal Q_C(A).$
		Since the solution corresponding to $Q_k$ is $X_k=A^{-1}Q_k=A^kXA^{-k},$
		it follows that $\mathcal O_A(X)\subseteq A^{-1}\mathcal Q_C(A).$
	\end{proof}
	
	Thus the parametrization $X\leftrightarrow Q=AX$ is equivariant with
	respect to conjugation by powers of $A$.
	
	\begin{theorem}\label{Centralizer_orbit}
		Let $A\in \P_n$. Then the centralizer $C_{\P_n}(A)$ acts on
		$\mathcal S_P(A)$ by conjugation,
		$$
		B\cdot X=BXB^{-1},
		\qquad B\in C_{\P_n}(A),\quad X\in\mathcal S_P(A).
		$$
		Moreover, for every $B,C\in C_{\P_n}(A)$,
		$$
		B\bigl(A^{-1}\mathcal Q_C(A)\bigr)B^{-1}=A^{-1}\mathcal Q_{BCB^{-1}}(A).
		$$
		Thus the centralizer permutes the solution sectors according to conjugation of their indices in $C_{\P_n}(A)$.
	\end{theorem}
	
	\begin{proof}
		Let $B\in C_{\P_n}(A)$ and $X\in\mathcal S_P(A)$. Since $AB=BA$,
		$$
			A(BXB^{-1})A=B(AXA)B^{-1}=B(XAX)B^{-1}=(BXB^{-1})A(BXB^{-1}),
		$$
		so $BXB^{-1}\in\mathcal S_P(A)$.
		
		Finally, let $X\in A^{-1}\mathcal Q_C(A),\;Q=AX.$ Then $Q^3=(QA)^2=C.$
		For $X_B=BXB^{-1},$ the corresponding parameter is $Q_B=AX_B=BQB^{-1},$
		and therefore $Q_B^3=(Q_BA)^2=BCB^{-1}.$
		Hence
		$$
		B\bigl(A^{-1}\mathcal Q_C(A)\bigr)B^{-1} \subseteq A^{-1}\mathcal Q_{BCB^{-1}}(A).
		$$
		Applying the same argument to $B^{-1}$ yields the reverse inclusion,
		and thus
		$$
		B\bigl(A^{-1}\mathcal Q_C(A)\bigr)B^{-1}
		=A^{-1}\mathcal Q_{BCB^{-1}}(A).
		$$
	\end{proof}
	
	\begin{corollary}
		Let $A\in \P_n$ and suppose that $C_{\P_n}(A)$ is Abelian.
		Then each orbit of the conjugation action of $C_{\P_n}(A)$ on
		$\mathcal S_P(A)$ is contained in a single solution sector.
		
		More precisely, if $X\in\mathcal S_P(A),\;Q=AX$ and $Q^3=(QA)^2=C,$
		then
		$$
		\mathcal O_{C_{\P_n}(A)}(X)\subseteq A^{-1}\mathcal Q_C(A).
		$$
	\end{corollary}
	
	\begin{proof}
		If $B,C\in C_{\P_n}(A)$ and the centralizer is Abelian, then $BCB^{-1}=C.$
		The assertion therefore follows immediately from the preceding theorem.
	\end{proof}
	
	\begin{corollary}
		Let $A\in \P_n$ be a permutation matrix corresponding to an $n$-cycle (one cycle of length $n$). Then
		$C_{\P_n}(A)=\langle A\rangle=\{I,A,\ldots,A^{n-1}\}.$
		Consequently,
		$$
		\mathcal S_P(A)=\bigcup_{k=0}^{n-1}A^{-1}\mathcal Q_{A^k}(A),
		$$
		where $\mathcal Q_{A^k}(A)=\{Q\in \P_n:Q^3=(QA)^2=A^k\}.$
		
		Thus, $X\in\mathcal S_P(A)$ if and only if, for the uniquely determined
		$Q=AX$, there exists a unique $k\in \mathbb{Z}_n$ such that $Q^3=(QA)^2=A^k.$
		In this case,
		$$
		X=A^{-1}Q=QAQ^{-1}.
		$$
	\end{corollary}
	
	\begin{proof}
		Since $A$ corresponds to an $n$-cycle, its centralizer in $\P_n$ is the cyclic group
		generated by $A$, that is, $C_{\P_n}(A)=\langle A\rangle.$
		Hence every element $C\in C_{\P_n}(A)$ is uniquely of the form
		$C=A^k,\; k\in \mathbb{Z}_n.$ The result now follows immediately from the preceding theorem.
	\end{proof}
	
	For an permutation matrix $A$ corresponding to an $n$-cycle, we refer to
	$$
	\mathcal Q_{A^k}(A)=\{Q\in \P_n:Q^3=(QA)^2=A^k\}
	$$
	as the $k$-th $Q$-sector, and to
	$$
	\mathcal S_{A^k}(A)=A^{-1}\mathcal Q_{A^k}(A)
	$$
	as the corresponding $k$-th solution sector.
	
	The sector $k=0$ is particularly convenient. Indeed, $Q^3=(QA)^2=I$
	implies that $Q$ consists only of cycles of lengths $1$ and $3$, whereas
	$QA$ consists only of cycles of lengths $1$ and $2$. However, this sector does not contain all non-commuting solutions in general, as shown in the following example.
	
	\begin{example}
		Let $A\in \P_{12}$ be a permutation matrix corresponding to a $12$-cycle $\pi_A$. Consider the permutation $\pi_Q=(1\,7)(2\,12\,10\,8\,6\,4)(3\,9)(5\,11)$. Then $\pi_Q^3=(\pi_Q \pi_A)^2=\pi_A^6\neq I$. Hence the corresponding permutation matrix $Q$ belongs to the $6$-th $Q$-sector. The corresponding solution matrix $X=QAQ^{-1}$ is associated with the permutation 
		$$\pi_X=\pi_Q \pi_A \pi_Q^{-1}=(1\,6\,3\,8\,5\,10\,7\,12\,9\,2\,11\,4).$$
		Thus $X$ is a non-commuting solution in the $6$-th solution sector.
		
		An exhaustive search for $n\leq 12$ shows that $n=12$ is the smallest order for which a non-commuting solution occurs outside the $0$-th sector. \hfill$\clubsuit$
	\end{example}
	
	\begin{theorem}\label{constr_noncom_0}
		Let $A\in\P_n$ be a permutation matrix corresponding to an $n$-cycle. For every $n\geq 4$ there exists an explicitly determined permutation matrix $Q\neq A^2$ such that $Q^3=(QA)^2=I$. Consequently, $X=A^{-1}Q=QAQ^{-1}$ is a non-commuting permutation solution of $AXA=XAX$. Conversely, if $n\leq 3$, every permutation solution commutes with $A$.
	\end{theorem}
	
	In other words, a non-commuting permutation solution exists if and only if $n\geq 4$, and for every $n\geq 4$ such a solution can be explicitly generated from the matrix $Q$ in the $0$-th $Q$-sector.
	
	\begin{proof}
		Let $\pi_A=(1\,2\,\dots\,n)$ be the permutation corresponding to $A$. According to the residue class of $n$ modulo $3$, we construct $q\in S_n$ such that $q^3=(q\pi_A)^2=I$. and let $Q\in\P_n$ be the corresponding permutation matrix. The corresponding permutation matrix $Q\in\P_n$ then satisfies $Q^3=(QA)^2=I.$ All cycle products below consist of pairwise disjoint cycles.\\
		
		\noindent\textbf{Case 1: $n=3m+1,\;m\geq 1$.} Define 
		$$q=\prod_{j=1}^{m}(2j,\ n-j+1,\ 2j+1)=(2,n,3)(4,n-1,5)\cdots(2m,2m+2,2m+1).$$ 
		Since $q$ is a product of pairwise disjoint $3$-cycles, $q^3=I$. Acting successively with $\pi_A$ and $q$ on the points $1,...,n$, one obtains 
		$$q\pi_A=(1,n)\prod_{j=1}^{m-1}(2j+1,n-j),$$ 
		where all points not occurring in the displayed transpositions are fixed. Thus $q\pi_A$ is an involution and hence $(q\pi_A)^2=I.$\\
		
		\noindent \textbf{Case 2: $n=3m+2,\;m\geq 1$.} Define 
		$$q=\prod_{j=1}^{m}(2j,\ n-j+1,\ 2j+1).$$
		Again, $q$ is a product of pairwise disjoint $3$-cycles, so $q^3=I.$ Moreover, 
		$$q\pi_A = (1,n) \left( \prod_{j=1}^{m-1}(2j+1,n-j) \right) (2m+1,2m+2),$$ 
		with all remaining points fixed. Therefore $q\pi_A$ is an involution, and $ (q\pi_A)^2=I.$\\
		
		\noindent \textbf{Case 3: $n=3m,\;m\geq 2$.} Define 
		$$q= \left( \prod_{j=1}^{m-2} (2j,\ n-j+1,\ 2j+1) \right) (2m-2,\ 2m+2,\ 2m).$$ 
		The displayed $3$-cycles are pairwise disjoint. Consequently, $q^3=I.$ Furthermore, 
		$$q\pi_A = (1,n) \left( \prod_{j=1}^{m-2}(2j+1,n-j) \right) (2m-2,2m-1)(2m,2m+1),$$ 
		where all points not appearing in the displayed transpositions are fixed. Hence $q\pi_A$ is an involution and $(q\pi_A)^2=I.$ 
		
		Thus, for every $n\geq 4$, the permutation $q$ satisfies $q^3=(q\pi_A)^2=I.$ Let $Q$ be its permutation matrix. Then $Q^3=(QA)^2=I.$ By the Theorem \ref{Centralizer_decomposition}, the matrix $X=A^{-1}Q$ is a permutation solution of $AXA=XAX.$ The same theorem also yields $X=QAQ^{-1}.$
		
		It remains to prove that $X$ does not commute with $A$. Suppose, to the contrary, that $AX=XA.$ Since $X$ also satisfies $AXA=XAX$, we obtain $A^2X=AX^2$, and the invertibility of $A$ and $X$ gives $A=X.$ It would then follow from $Q=AX$ that $Q=A^2.$ This is impossible: the permutation $q_n$ fixes the point $1$, whereas $\pi_A^2$ has no fixed points for $n\geq 4$. Therefore $AX\neq XA.$
		
		Conversely, suppose that $n\leq 3$. For $n=1$, the statement is trivial. For $n=2$, there is only one $2$-cycle, namely $\pi_A$, so every permutation solution conjugate to $\pi_A$ must equal $\pi_A$. For $n=3$, the only $3$-cycles are $\pi_A$ and $\pi_A^{-1}$. The permutation $\pi_A^{-1}$ is not a solution, since $\pi_A \pi_A^{-1}\pi_A=\pi_A, \; \pi_A^{-1}\pi_A \pi_A^{-1}=\pi_A^{-1}$ and $\pi_A\neq \pi_A^{-1}$. Hence the only permutation solution for $n\leq 3$ is the commuting solution $\pi_X=\pi_A$. Therefore, a non-commuting permutation solution exists precisely when $n\geq 4.$
	\end{proof}
	
	\begin{remark}
		If $A\in\P_n$ corresponds to an $n$-cycle, then $\operatorname{ord}(A)=n$.
		Hence every $A$-orbit has cardinality dividing $n$.
		Moreover, by Corollary \ref{cor_orbits}, each $A$-orbit is contained in
		a single $k$-th solution sector for some $k\in\mathbb{Z}_n$.
	\end{remark}
	
	\begin{example}\label{Ex_8}
		Let us find a non-commuting permutation solution $X_0\in\P_8$ (corresponding to $\pi_{X_0}\in S_8$) to the YBME $AXA=XAX$, where $A\in\P_8$ is a permutation matrix corresponding to a $8$-cycle $\pi_A$. Since $8=3\cdot 2+2$, Theorem \ref{constr_noncom_0}, case 2, gives
		$$q=\prod_{j=1}^{2}(2j, 9-j, 2j+1)=(2,8,3)(4,7,5)(1)(6),$$
		hence $\pi_{X_0}=q\pi_A q^{-1}=(1,8,2,7,4,6,5,3).$ In matrix form,
		$$X_0=\left[\begin{array}{cccccccc}
			0 & 0 & 1 & 0 & 0 & 0 & 0 & 0 \\
			0 & 0 & 0 & 0 & 0 & 0 & 0 & 1 \\
			0 & 0 & 0 & 0 & 1 & 0 & 0 & 0 \\
			0 & 0 & 0 & 0 & 0 & 0 & 1 & 0 \\
			0 & 0 & 0 & 0 & 0 & 1 & 0 & 0 \\
			0 & 0 & 0 & 1 & 0 & 0 & 0 & 0 \\
			0 & 1 & 0 & 0 & 0 & 0 & 0 & 0 \\
			1 & 0 & 0 & 0 & 0 & 0 & 0 & 0
		\end{array}\right].$$
		Since $X_0$ is a $8$-cycle, its orbit must be of the order $8$. Hence, we obtain seven more non-commuting solutions:
		\begin{align*}
			\pi_{X_1} &=\pi_A \pi_{X_0}\pi_A^{-1}=(1,3,8,5,7,6,4,2),\\
			\pi_{X_2} &=\pi_A \pi_{X_1}\pi_A^{-1}=(1,6,8,7,5,3,2,4),\\
			\pi_{X_3} &=\pi_A \pi_{X_2}\pi_A^{-1}=(1,8,6,4,3,5,2,7),\\
			\pi_{X_4} &=\pi_A \pi_{X_3}\pi_A^{-1}=(1,7,5,4,6,3,8,2),\\
			\pi_{X_5} &=\pi_A \pi_{X_4}\pi_A^{-1}=(1,3,2,8,6,5,7,4),\\
			\pi_{X_6} &=\pi_A \pi_{X_5}\pi_A^{-1}=(1,7,6,8,5,2,4,3),\\
			\pi_{X_7} &=\pi_A \pi_{X_6}\pi_A^{-1}=(1,6,3,5,4,2,8,7).
		\end{align*}
		This defines one orbit of the action $X\mapsto A^kXA^{-k}$. Continuing this process we obtain $25$ solutions in total:
	 $$\begin{aligned}&(1\,3\,8\,5\,7\,6\,4\,2),\quad (1\,3\,8\,6\,5\,7\,4\,2),\quad (1\,4\,3\,8\,6\,7\,5\,2),\quad (1\,6\,4\,5\,3\,8\,7\,2),\\
	 &(1\,7\,4\,6\,5\,3\,8\,2)\quad (1\,7\,5\,4\,6\,3\,8\,2),\quad (1\,2\,8\,5\,4\,7\,6\,3),\quad (1\,6\,5\,8\,7\,4\,2\,3),\\ 
	 &(1\,6\,8\,7\,5\,2\,4\,3),\quad  (1\,7\,6\,8\,5\,2\,4\,3),\quad (1\,8\,2\,7\,4\,6\,5\,3),\quad (1\,8\,2\,7\,5\,4\,6\,3),\\
	&(1\,3\,2\,8\,5\,7\,6\,4),\quad (1\,3\,2\,8\,6\,5\,7\,4),\quad (1\,6\,8\,7\,5\,3\,2\,4),\quad (1\,7\,6\,8\,5\,3\,2\,4),\\
	&(1\,7\,8\,6\,3\,2\,5\,4),\quad  (1\,8\,3\,2\,7\,5\,6\,4),\quad (1\,8\,5\,3\,4\,2\,7\,6),\quad (1\,6\,3\,5\,4\,2\,8\,7),\\
	&(1\,8\,6\,3\,5\,4\,2\,7),\quad  (1\,6\,4\,3\,5\,2\,8\,7),\quad (1\,8\,6\,4\,3\,5\,2\,7),\quad(1\,7\,4\,3\,6\,5\,2\,8).\end{aligned}$$
		
	Read from left to right, row after row, denote the obtained permutations as $X_1$, $\ldots,$ $X_{24}$. There are precisely $4$ non-commuting orbits under the action $X\mapsto A^k XA^{-k}$ for $A=(1\,2\,3\,4\,5\,6\,7\,8)$, of the sizes $4$ or $8$:
		\begin{enumerate}
			\item $\mathcal O_1=\{X_1, X_{15}, X_{23}, X_6, X_{14}, X_{10}, X_{20}, X_{11}\}$;
			\item $\mathcal O_2=\{X_2, X_{16}, X_{22}, X_{12}\}$;
			\item $\mathcal O_3=\{X_3, X_{17}, X_{24}, X_7, X_8, X_{19}, X_4, X_{18}\}$;
			\item $\mathcal O_4=\{X_5, X_{13}, X_9, X_{21}\}$.
		\end{enumerate}
			\hfill$\clubsuit$
	\end{example}
	
	\section{DS Solutions for a Cyclic Permutation $A$}\label{cyclic}
	
	Throughout this section we assume that $A$ is a permutation matrix of order $n$, consisting of only one cycle of the length $n$. In what follows we obtain all doubly stochastic solutions for the said cyclic matrix $A$, regardless of whether they (the solutions) are permutation matrices or proper doubly stochastic solutions.
	
	By the virtue of \eqref{similarity}, without the loss of generality we can rearrange the base vectors in such a fashion that the cyclic matrix $A$ is given as:
	
	\begin{equation}
		\label{diagA}
		A=\left[\begin{array}{cccccc}
			0 & 0 & 0 & \ldots & 0 & 1\\
			1 & 0 & 0 &\ldots & 0 & 0\\
			0 & 1 & 0 & \ldots & 0 & 0 \\
			\vdots & \vdots& \vdots& \vdots& \vdots &\vdots\\
			0 & 0 & 0 &\ldots & 1 & 0
		\end{array}\right].
	\end{equation}
	For easier notation, for any matrix $Y$, we denote by $Y(i,j)$ or by $y_{i,j}$ its position in the $i-$th row and $j-$th column. The action of the matrix $A$ on a matrix $Y$ shifts its row (resp. column) order when acting as $Y\mapsto AY$, or, respectively, as $Y\mapsto YA$. Therefore we denote by $Y(i+k, j+\ell)$ the $\mod n-$operations $Y\left(i+_nk,j+_n\ell\right)$, to simplify the indexing, where the $\mod n-$ set is chosen to be $\{1,\ldots,n\}$, in accordance with matrix entries' coordinates.
	
	We start with the following auxiliary observations, which will be useful in the later text.

	\begin{proposition}\label{YAcirc}
		Let $Y=\left[y_{i,j}\right]_{n\times n}$ be any square $n-$dimensional matrix. For the permutation $A$ given via \eqref{diagA} above, it follows that for every $i,j\in\{1,\ldots,n\}$:
		\begin{equation}
			\label{AY} AY(i,j)=Y(i-1,j), \end{equation}
		\begin{equation}
			\label{YA}
			YA(i,j)=Y(i,j+1),
		\end{equation}
		\begin{equation}
			\label{AYA}
			AYA(i,j)=Y(i-1,j+1),
		\end{equation}
		\begin{equation}
			\label{YAY}
			YAY(i,j)=\sum_{s=1}^n Y(i,s+1)\cdot Y(s,j).
		\end{equation}
	\end{proposition}
	
	The above Proposition \ref{YAcirc} is directly verifiable. It is clear that a matrix $Y$ is a solution to \eqref{YBME} if and only if the equalities \eqref{AYA}--\eqref{YAY} are mutually equal for every $i,j\in\{1,\ldots,n\}$, that is, if and only if
	\begin{equation}
		\label{condtij}
		Y(i-1,j+1)=\sum_{s=1}^n Y(i,s+1)\cdot Y(s,j)
	\end{equation} 
	holds for every $i,j\in\{,1\ldots,n\}$. This setting gives several necessary properties for the d. s. solutions.
	\subsection{Necessary conditions}
	Throughout this subsection it is understood that $A$ is a cyclic matrix provided as \eqref{diagA} above, and that $X$ is a doubly stochastic solution to \eqref{YBME}. 
	
	\begin{lemma}
		Let $X$ be a doubly stochastic solution to \eqref{YBME}. There exists a positive number $R_X\in(0,1]$ such that for every $1\leq i\leq n$:
		\begin{equation}
			\label{R} R_X=\max_{1\leq j\leq n} X(i,j)=\max_{1\leq j\leq n} X(j,i).
		\end{equation}
		In other words, the largest entry in each row and in each column of $X$ is the same.
	\end{lemma}
	\begin{proof} For each $j\in\{1,\ldots,n\}$ denote by
		\begin{equation}\label{Rj}R_j=\max_{1\leq i\leq n} X(i,j),
		\end{equation}
		i.e., the largest value in the $j-$th column of the matrix $X$. Obviously each $R_j\in(0,1]$ because $X$ is a doubly stochastic matrix. Fix an arbitrary $j_0\in\{1,\ldots,n\}$, and choose an $i_0\in\{1,\ldots,n\}$ in such a way that the index $i_0-1$ ($\mod n$) satisfies that $X(i_0-1,j_0+1)=R_{j_0+1}$. Then by \eqref{AYA} and \eqref{YAY}, one has that
		\begin{eqnarray}\label{Rj0}
			\begin{aligned}
				R_{j_0+1}&=X(i_0-1,j_0+1)=AXA(i_0,j_0)=XAX(i_0,j_0)=\\
				&=\sum_{s=1}^n X(i_0,s+1)X(s,j_0)\leq \left(\sum_{s=1}^n X(i_0,s+1)\right)R_{j_0}=\\
				&=R_{j_0}.
		\end{aligned}\end{eqnarray}
		Since $j_0$ was arbitrary, it follows that there exists a (clearly positive) $R\in(0,1]$ such that
		\begin{equation}
			\label{Rjs}
			R=R_1=\ldots=R_n.
		\end{equation}
		An analogous procedure goes for the row-wise maximums: for each $i$ denote by 
		\begin{equation}
			\label{R'i}
			R'_i=\max_{1\leq j\leq n}X(i,j).
		\end{equation}
		As before, it follows that $R'_i\in(0,1]$. For any fixed $i_0$, choose $j_0$ so that $X(i_0-1,j_0+1)=R'_{i_0-1}$. Then 
		\begin{eqnarray}
			\label{Ri0}
			\begin{aligned}
				R'_{j_0+1}&=X(i_0-1,j_0+1)=AXA(i_0,j_0)=XAX(i_0,j_0)=\\
				&=\sum_{s=1}^n X(i_0,s+1)X(s,j_0)\leq R'_{i_0}\left(\sum_{s=1}^n X(s,j_0)\right)=\\
				&=R'_{i_0}.
		\end{aligned}\end{eqnarray}
		The estimate \eqref{Ri0} is true for every $i_0$, therefore there exists a (clearly positive) $R'\in(0,1]$ such that
		\begin{equation}
			\label{Ris}
			R'=R'_1=\ldots=R'_n.
		\end{equation}
		If $R\geq R'$, then let $s$ be (one of the positions) in which $X(1,s)=R'$. But then there exists a $k\in\{1,\ldots, n\}$ such that
		$X(k,s)=R$. Finally, in the $k-$th row, there exists a position $\ell\in\{1,\ldots,n\}$ such that $X(k,\ell)=R'$. In summary: 
		$$R'=X(1,s)\leq R=X(k,s)\leq\max_{1\leq y\leq n}X(k,y)=X(k,\ell)=R',$$
		concluding that $R=R'$. Define $R_X:=R\equiv R'$. 
	\end{proof}
	In the forthcoming text, it is understood that for a given doubly stochastic solution $X$ the value $R_X$ is defined as \eqref{R}. Accordingly, for each $j$ we introduce the set 
	\begin{equation}
		\label{Ij}I_{j}=\{i\in\{1,\ldots,n\}: X(i,j)=R_X\}.
	\end{equation}
	Clearly $I_j$ is nonempty, and we denote by $|I_j|$ its cardinality. By $I^c_j$ we denote the set $\{1,\ldots,n\}\setminus I_j$, which can be empty. Analogously, we define
	\begin{equation}
		\label{Ji}J_{i}=\{j\in\{1,\ldots,n\}: X(i,j)=R_X\},
	\end{equation}
	and the set $J_i$ is nonempty, and again by $|J_i|$ we denote its cardinality. On the other hand,  its complement $J_i^c=\{1,\ldots,n\}\setminus J_i$ may be empty. In fact, the following lemma holds:
	
	\begin{lemma}\label{emptycomplements} Let $X$ be a doubly stochastic solution to \eqref{YBME}, with the constant $R_X$ provided as \eqref{R}. Let the sets $I_j$ and $J_i$ be provided via \eqref{Ij} and \eqref{Ji}, respectively, for each $1\leq i,j\leq n$. The following statements are equivalent:
		\begin{itemize}
			\item[(a)] There exists an $i_0$ such that $J_{i_0}=\{1,\ldots,n\}$.
			\item[(b)] There exists a $j_0$ such that $I_{j_0}=\{1,\ldots,n\}$.
			\item[(c)] $X=U_n$.
		\end{itemize}
	\end{lemma}
	\begin{proof}
		If $X=U_n$ then $R_X=\frac{1}{n}$ and it is attained at every $(i,j)$ throughout the matrix $X$.
		
		Conversely, assume there exists a constant row $X(i_0,\cdot)$ for some index $i_0$, such that $X(i_0,j)=R_X$ for every $j$. Then $R_X=\frac{1}{n}$, and  $$X(i_0,\cdot)=\left(\frac{1}{n},\frac{1}{n},\ldots,\frac{1}{n}\right).$$
		If $X\neq U_n$, there exists a row $X(i_1,\cdot)$ in which not all entries are equal to the maximum $\frac{1}{n}$, so $J_{i_1}^c$ is nonempty. But in the said row we have:
		$$\begin{aligned}
			1&=\sum_{j=1}^n X(i_1,j)=\sum_{j\in J_{i_1}}\frac{1}{n}+\sum_{j\notin J_{i_1}}X(i_1,j)=\frac{|J_{i_1}|}{n}+\sum_{j\notin J_{i_1}}X(i_1,j)<\\
			&<\frac{|J_{i_1}|}{n}+\sum_{j\notin J_{i_1}}\frac{1}{n}=\frac{|J_{i_1}|+n-|J_{i_1}|}{n}=\frac{n}{n}=1,
		\end{aligned}$$
		which is impossible. Therefore $(a)\Leftrightarrow(c)$, and $(b)\Leftrightarrow(c)$ is analogous.
	\end{proof}
	\begin{theorem}\label{udmus}
		Let $A$ be given as \eqref{diagA}, and let $X$ be a doubly stochastic solution to \eqref{YBME}. Then there exists an $n_0\in\{1,\ldots,n\}$ such that $X\in\U_n(n_0)$. Moreover, the value $R_X$ from \eqref{R} is exactly $R_X=\frac{1}{n_0}$.
	\end{theorem}
	
	\begin{proof}
		If $X$ is regular then it is a permutation matrix, and $n_0=R_X=1$.
		
		If $X=U_n$, then $n_0=n$, and by Lemma \ref{emptycomplements} $R_X=\frac{1}{n}$. 
		
		Assume that $X$ is a proper doubly stochastic solution to \eqref{YBME} such that $X\neq U_n$. Let $(i_0,j_0)$ be a position such that $X(i_0-1,j_0+1)=R_X$. Since $I_{j_0}^c$ and $J_{i_0}^c$ are nonempty, we have
		\begin{eqnarray}\label{RXIj0}
			\begin{aligned}
				R_X&=X(i_0-1,j_0+1)=\sum_{s=1}^nX(i_0,s+1)X(s,j_0)=\\
				&=\sum_{s\in I_{j_0}}X(i_0,s+1)R_X+\sum_{s\notin I_{j_0}}X(i_0,s+1)X(s,j_0)\Leftrightarrow\\
				&\left(1-\sum_{s\in I_{j_0}}X(i_0,s+1)\right)R_X=\sum_{s\notin I_{j_0}}X(i_0,s+1)X(s,j_0)\Leftrightarrow\\
				&\left(\sum_{s\notin I_{j_0}}X(i_0,s+1)\right)R_X=\sum_{s\notin I_{j_0}}X(i_0,s+1)X(s,j_0).
			\end{aligned}
		\end{eqnarray}
		By definition, each $X(s,j_0)$ is strictly smaller than $R_X$ whenever $s\in I_{j_0}^c$, therefore the latter equality is possible if and only if 
		\begin{equation}\label{Ij0}
			\left(\forall s\in I_{j_0}^c\right)\quad X(i_0,s+1)=0.
		\end{equation}
		At the same time, \eqref{RXIj0} reads 
		\begin{eqnarray}\label{RXJi0}
			\begin{aligned}
				R_X&=X(i_0-1,j_0+1)=\sum_{s=1}^nX(i_0,s+1)X(s,j_0)=\\
				&=\sum_{s+1\in J_{i_0}}R_X X(s,j_0)+\sum_{s+1\notin J_{i_0}}X(i_0,s+1)X(s,j_0)\Leftrightarrow\\
				&\left(1-\sum_{s+1\in J_{i_0}}X(s,j_0)\right)R_X=\sum_{s+1\notin J_{i_0}}X(i_0,s+1)X(s,j_0)\Leftrightarrow\\
				&\left(\sum_{s+1\notin J_{i_0}}X(s,j_0)\right)R_X=\sum_{s+1\notin J_{i_0}}X(i_0,s+1)X(s,j_0).
			\end{aligned}
		\end{eqnarray}
		By the same argument as before, $X(i_0,s+1)<R_X$, whenever $s+1\in J_{i_0}^c$,  therefore
		\begin{equation}\label{Ji0} \left(\forall s+1\in J_{i_0}^c\right)\quad  X(s,j_0)=0.
		\end{equation}
		By introducing the set
		\begin{equation}
			\label{Ji01}J_{i_0}^c-1=\{k-1: k\in J_{i_0}^c\}\equiv \{k+_n (n-1): k\in J_{i_0}^c\},
		\end{equation}
		where again the operations modulo $n$ take values in the set $\{1,\ldots,n\}$, it follows that
		$$s+1\in J_{i_0}^c\Leftrightarrow s\in J_{i_0}^c-1,$$
		and \eqref{Ji0} reads \begin{equation}
			\label{Ji0-1} \left(\forall s\in J_{i_0}^c-1\right)\quad  X(s,j_0)=0.
		\end{equation}
		By definition, whenever $s\in I_{j_0}^c$, we have $X(s,j_0)<R_X$, therefore from \eqref{Ji0-1} we conclude that
		\begin{equation}
			J_{i_0}^c-1\subset I_{j_0}^c.
		\end{equation}
		And from \eqref{Ij0} we get
		$$\left(\forall s\in J_{i_0}^c-1\right)\quad X(i_0,s+1)=0\Leftrightarrow \left(\forall s+1\in J_{i_0}^c\right)\quad X(i_0,s+1)=0.$$
		In other words, 
		\begin{equation}\label{Xi0j}
			X(i_0,j)=\begin{cases} R_X,\quad j\in J_{i_0},\\0,\quad j\in J_{i_0}^c.
			\end{cases}
		\end{equation}
		This concludes that $n_0$ exists and is equal to $n_0:=|J_{i_0}|$, and $R_X=\frac{1}{n_0}$. Since $R_X$ is uniform throughout the matrix $X$, and $i_0$ was an arbitrarily chosen row, the proof is complete.
	\end{proof}
	Whenever $X$ a doubly stochastic solution for a cyclic matrix $A$, we denote by $X(n_0)$ the fact that $X\in\U_n(n_0)$. The value $R_{X(n_0)}$ from \eqref{R} is then simply $R_{X(n_0)}=\frac{1}{n_0}$.  
	\begin{corollary}\label{cardinality} Let $X(n_0)$ be a doubly stochastic solution to \eqref{YBME}. Then for every $i_0,j_0\in\{1,\ldots,n\}$:
		\begin{equation}
			\label{} I_{j_0}=\left\{i\in\{1,\ldots,n\}: X(i,j_0)=\frac{1}{n_0}\right\},\quad |I_{j_0}|=n_0,
		\end{equation}
		\begin{equation}
			\label{} I_{j_0}^c=\left\{i\in\{1,\ldots,n\}: X(i,j_0)=0\right\},\quad |I_{j_0}^c|=n-n_0,
		\end{equation}
		\begin{equation}
			\label{} J_{i_0}=\left\{j\in\{1,\ldots,n\}: X(i_0,j)=\frac{1}{n_0}\right\},\quad |J_{i_0}|=n_0,
		\end{equation}
		\begin{equation}
			\label{} J_{i_0}^c=\left\{j\in\{1,\ldots,n\}: X(i_0,j)=0\right\},\quad |J_{i_0}^c|=n-n_0.
		\end{equation}
	\end{corollary}
	In analogy to \eqref{Ji01} we introduce the set
	\begin{equation}
		\label{Ij1}I_{j}-1=\{k-1: k\in I_j\}\equiv \{k+_n (n-1): k\in I_j\}.
	\end{equation}
	
	\begin{corollary}\label{samesets}
		Let $X(n_0)$ be a doubly stochastic solution to \eqref{YBME}. For arbitrary  $i_0,j_0\in\{1,\ldots,n\}$ the following statements are equivalent:
		\begin{itemize}
			\item[(a)] $X(i_0,j_0)=\frac{1}{n_0}$.
			\item[(b)] $i_0\in I_{j_0}$.
			\item[(c)] $j_0\in J_{i_0}$.
			\item[(d)] $I_{(j_0-1)}-1=J_{i_0+1}$.
		\end{itemize}
		
	\end{corollary}
	\begin{proof} $(a)\Leftrightarrow(b)\Leftrightarrow(c)$ is obvious. The equivalence $(a)\Leftrightarrow(d)$ follows from \eqref{condtij}:
		\begin{equation}
			\label{Xi0j0}
			X(i_0,j_0)=\sum_{s=1}^n X(i_0+1,s)\cdot X(s-1,j_0-1)=\sum_{s}\frac{1}{n_0^2},\end{equation}
		where the last sum runs over the intersection $I_{(j_0-1)}-1\cap J_{i_0+1}$ (all other entries are zeros). Therefore 
		$|I_{(j_0-1)}-1\cap J_{i_0+1}|=n_0$, i.e. $(d)$ holds if and only if $X(i_0,j_0)=\frac{1}{n_0}$.
	\end{proof}
	\begin{corollary}\label{disjoint}
		Let $X(n_0)$ be a doubly stochastic solution to \eqref{YBME}. For arbitrary  $i_0,j_0\in\{1,\ldots,n\}$ the sets $I_{(j_0-1)}-1$ and $J_{i_0+1}$ either coincide or are disjoint.
	\end{corollary}
	\begin{proof}
		Assume that for some $i_0$ and $j_0$ the sets $I_{(j_0-1)}-1$ and $J_{i_0+1}$ have a nonempty intersection $\mathfrak{I}_0$ such that $\mathfrak{I}_0\subset J_{i_0+1}$. Then \eqref{Xi0j0} applies once again, giving
		$$0<\sum_{s\in\mathfrak{I}_0}\frac{1}{n_0^2}=\sum_{s=1}^n X(i_0+1,s)\cdot X(s-1,j_0-1)=X(i_0,j_0)<\frac{1}{n_0},$$ 
		which is impossible.
	\end{proof}
	\begin{theorem}\label{divisors}
		Let $X(n_0)$ be a doubly stochastic solution to \eqref{YBME}. The number $d:=R_{X(n_0)}n=\frac{n}{n_0}$ is a natural number, i.e., $n_0|n$.
	\end{theorem}
	\begin{proof} Specially, when $n_0=1$, then $X$ is a permutation matrix and the claim is automatically true. 
		In that case $d=n$.

		On the other hand, when $n_0=n$, then $d=1$, $X=U_n$ is the uniform distribution matrix, and the claim is also true. 
		
		We proceed to analyze the remaining cases for $n_0\in\{1,\ldots,n-1\}$. Let $$X(1,j_1)=X(1,j_2)=\ldots=X(1,j_{n_0})=\frac{1}{n_0},$$ for some $j_1,\ldots,j_{n_0}\in\{1,\ldots,n\}$. Then by Corollary \ref{samesets}
		$$I_{j_1-1}=I_{j_2-1}=\ldots=I_{j_{n_0}-1}=J_2+1,$$
		implying that the columns are entry-wise the same:
		$$X(\cdot, j_1-1)=X(\cdot, j_2-1)=\ldots=X(\cdot,j_{n_0}-1).$$
		Moreover, if there existed an $j_s\in\{1,\ldots,n\}\setminus\{j_1,\ldots,j_{n_0}\}$, such that the column $X(\cdot, j_s-1)$ were entry-wise the same as the $X(\cdot, j_1-1)$, then it would be entry-wise identical with all of them, producing
		$$X(1,j_{s})=\frac{1}{n_0}.$$
		But
		$$1=\sum_{j\in\{j_1,\ldots,j_{n_0}\}}X(1,j)+X(1,j_s)=\frac{n_0+1}{n_0}$$
		which is impossible. Similarly, by Corollary \ref{disjoint}, there cannot be a column $X(\cdot,j_s-1)$ (again $j_s\in\{1,\ldots,n\}\setminus\{j_1,\ldots,j_{n_0}\}$), such that $I_{j_s-1}\cap J_{2}+1\neq\emptyset$, therefore 
		there are exactly $n_0$ columns in the matrix $X$ which have the value $\frac{1}{n_0}$ positioned at the same coordinates, say $i_1,\ldots,i_{n_0}$:
		$$X(i,j)=\frac{1}{n_0},\quad (i,j)\in\{i_1,\ldots,i_{n_0}\}\times\{j_1-1,\ldots,j_{n_0}-1\},$$ 
		and none of the remaining columns $X(\cdot,j_s-1)$, $j_s\in \{1,\ldots,n\}\setminus\{j_1,\ldots,j_{n_0}\}$, can have the entry $\frac{1}{n_0}$ at any of the positions $i_1,\ldots,i_{n_0}$, and, analogously, none of the remaining rows $X(i_\ell,\cdot)$, $i_\ell\in\{1,\ldots,n\}\setminus\{i_1,\ldots,i_{n_0}\}$ can have the entry $\frac{1}{n_0}$ at any of the positions $j_1-1,\ldots,j_{n_0}-1$.
		This procedure eliminated (occupied) exactly $n_0$ rows and exactly $n_0$ columns of the matrix $X$, and there are $n-n_0$ rows and $n-n_0$ columns remaining that need to be occupied. Repeating this procedure $[n/n_0]$ times, where $[\cdot]$ is the floor function, one eventually has $r:=n-[n/n_0]n_0$ remaining rows $X(k_1,\cdot)$, $\ldots,$ $X(k_r,\cdot)$ and r remaining columns $X(\cdot, m_1)$, $\ldots$, $X(\cdot, m_r)$,  all of which need to obey the rules from Corollary \ref{cardinality}. Notice that $0\leq r<n_0$ by construction. 
		
		Assume $r>0$. By Corollary \ref{disjoint}, it follows that 
		$$\left(\bigcup_{s=1}^rI_{m_s}\right)\bigcap\left(\bigcup_{j\in\{1,\ldots,n\}\setminus\{m_1,\ldots,m_r\}} I_{j}\right)=\emptyset$$
		and
		$$\left(\bigcup_{s=1}^rJ_{k_s}\right)\bigcap\left(\bigcup_{i\in\{1,\ldots,n\}\setminus\{k_1,\ldots,k_r\}} J_i\right)=\emptyset,$$
		giving that
		$$X(k_1,j_0)=\frac{1}{n_0}\Leftrightarrow j_0\in\{m_1,\ldots,m_r\}$$
		but then 
		$$1=\sum_{q=1}^n X(k_1,q)=\sum_{j_0\in\{m_1,\ldots,m_r\}}\frac{1}{n_0}=\frac{r}{n_0}<1,$$
		which is impossible. Therefore $r=0$ and $n_0|n$.
	\end{proof}
	
	\begin{theorem}\label{equidistant} Let $X(n_0)$ be a doubly stochastic solution to \eqref{YBME} and assume $A$ is provided as in \eqref{diagA}. Then the solution $X(n_0)$ has the entries $\frac{1}{n_0}$ equidistantly distributed throughout each row and column, with exactly $d-1$ zeros between two consecutive nonzero entries within the same row (resp. column).
	\end{theorem}
	\begin{proof} When $n_0\in\{1,n\}$ the claim is obvious and true (permutation matrices and the uniform distribution matrix $U_n$ possess this property). Assume that $1<n_0<n$, i.e, assume $n_0$ is a proper divisor of $n$.
		
		First notice, due to Corollary \ref{samesets} and Corollary \ref{disjoint}, that the nonzero entries are evenly distributed throughout rows if and only if they are evenly distributed throughout columns.  
		
		Let $J_1=\{j_1,\ldots, j_{n_0}\}$,  that is, $$X(,j_1)=X(,j_2)=\ldots=X(1,j_{n_0})=\frac{1}{n_0}.$$  
		If  $j_1$ and $j_2$ are adjacent modulo $n$, i.e., $j_2=j_1+_n1$, then by Proposition \ref{YAcirc} it follows that
		$$X(1,j_2)=X(2,1)\cdot X(n,j_1)+X(2,2)\cdot X(1,j_1)+\ldots+X(2,n)\cdot X(n-1,j_1).$$
		By assumption, $X(1,j_1)\neq0$, therefore the entry $X(2,2)$ must be nonzero as well, by Corollary \ref{samesets} $(d)$. On the other hand, conducting the same calculations for $X(1,j_1)$ in terms of Proposition \ref{YAcirc} one gets that
		$$\begin{aligned}
			X(1,j_1)&=X(2,1)\cdot X(n,j_1-1)+X(2,2)\cdot X(1,j_1-1)+\ldots+\\
			&+X(2,n)\cdot X(n-1,j_1-1).\end{aligned}$$
		The latter implies that $X(1,j_1-_n1)\neq0$, by the same argument from Corollary \ref{samesets} $(d)$.  Thus the premise that $$X(1,j_1)=X(1,j_1+1)=\frac{1}{n_0}$$ 
		implies that $X(1,j_1-1)=\frac{1}{n_0}$ as well. Continuing this process, one gets that all entries in the first row are equal to $\frac{1}{n_0}$, which, by  Lemma \ref {emptycomplements} implies that $n_0=n$, contradicting the premise that $n_0<n$. Consequently, the positions  $j_1,\ldots,j_{n_0}$ are not adjacent modulo $n$, since $j_1$ and $j_2$ were chosen arbitrarily from $J_1$. Notice that an analogous conclusion holds for the distribution of ${1}{n_0}$ along the same column: if there is a column of the solution $X(n_0)$ in which two nonzero entries are adjacent (mod $n$) then that entire column comprises out of the nonzero entry $\frac{1}{n_0}$, implying that $n=n_0$. 
		
		Therefore the set $J_2=\{j'_1,\ldots,j'_{n_0}\}$ is disjoint with the set $J_1$, where
		$$X(2,j'_1)=\ldots=X(2,j'_{n_0})=\frac{1}{n_0}.$$
		By \eqref{condtij} we have for every $j\in J_1$:
		$$\frac{1}{n_0}=X(1,j)=\sum_{j'\in J_2} X(2,j')\cdot X(j'-1,j-1),$$
		rendering $X(j'-1, j-1)=\frac{1}{n_0}$, for every $j\in J_1$, $j'\in J_2$.  Fix one $j'_0\in J_2$. Then in the row $t_0:=j'_0-1$, we have that the nonzero entries are positioned at $j-1$, for every $j\in J_1$. Clearly $t_0\neq 1$ because the nonzero entries cannot be adjacently positioned in the first row. By Corollary \ref{samesets} it follows that 
		$$I_{j_1-2}=I_{j_2-2}=\ldots=I_{j_{n_0}-2},$$
		equivalently, the columns $X(\cdot, j-2)$ are all entry-wise equal, for every $j\in J_1$. 
		
		If there is an $\ell\in\{1,\ldots,n_0\}$ such that $X(1,j_\ell)=X(1,j_\ell-2)=\frac{1}{n_0}$, then all the entries $j_1,\ldots,i_{n_0}$ must be positioned in the first row consecutively, with the constant distance of two positions (modulo $n$). This is because
		$$X(1,j_\ell+2)=X(1,j_\ell)=X(1,j_\ell-2)=\frac{1}{n_0}\Rightarrow j_\ell-2\in\{j_1,\ldots,j_{n_0}\}$$
		and so on. In that case $2\cdot n_0=n$, i.e., $d=2$ and all nonzero entries are evenly distributed throughout the first row. 
		
		Otherwise, if none of the positions $j_1,\ldots, j_{n_0}$ are positioned in this fashion, then we continue in the recursive manner described below. Precisely, observe the row $t_0+1$. If the nonzero entries (in the said row) are positioned at $s_1,\ldots,s_{n_0}$, it follows that the columns $X(\cdot; j_1-2),\ldots,X(\cdot; j_{n_0}-2)$ have the nonzero entries at positions $(s-1, j-2)$, for every $s\in J_{t_0+1}$ and every $j\in J_1$, because $X(t_0,j-1)=\frac{1}{n_0}$ for every $j\in J_1$ (again this follows from \eqref{condtij}).

		For an arbitrary (but fixed) $s_0\in J_{t_0+1}$, denote by $t_1:=s_0-1$. Then, in the row $t_1$, the nonzero entries are located at  positions $j-2$, for every $j\in J_1$. By applying \eqref{condtij} once again we conclude that the columns $X(\cdot; j-3)$ are all entry-wise equal, for every $j\in J_1$.
		
		Denote by $d_1\in\NN$, $d_1>1$, the difference $j_2-j_1$ modulo $n$, that is, $j_1+_nd_1=j_2$. Continuing the process described above, we eventually obtain the step $d_1$, such that the columns $X(\cdot; j-d_1)$ are all entry-wise equal, for every $j\in J_1$. Specially, for $j=j_2$, one gets $j_1=j_2-d_1$, and by Corollary \ref{samesets} and Corollary \ref{disjoint} it follows that
		$$j_2=j_1+d_1,\quad j_3=j_2+d_1,\quad,\ldots,\quad j_{n_0}+d_1=j_1,$$
		proving that the nonzero entries are evenly distributed through the first row. Since the first row was chosen arbitrarily, the same conclusion holds for every row of the matrix $X(n_0)$, and by the arguments mentioned above, the same is true for each column of the matrix $X(n_0)$. 
	\end{proof}

	\subsection{All d. s. solutions for $1-$cycle}
	In this subsection we characterize all d. s. solutions to \eqref{YBME}, where $A$ is assumed to be a one-cycle permutation provided via \eqref{diagA}. 
	
	For a fixed $d\in\NN$, we say that a permutation matrix $D\in\RR^{d\times d}$ belongs to the set $\Omega_d$, if and only if 
	\begin{eqnarray}\label{Dperm}\begin{aligned}
			&D\in\Omega_d\Leftrightarrow\textrm{$D$ is a permutation matrix such that}\\
			&(\forall i,j\in\{1,\ldots,d\})D(i,j)=1\Leftrightarrow\\
			&\left(\exists s_{(i,j)}\in\{1,\ldots,d\}\right)\left(D(i+_d1,s_{(i,j)})=1 \wedge D(s_{(i,j)}-_d1,j-_d1)=1\right).\end{aligned}
	\end{eqnarray}

	\begin{theorem}[All d. s. solutions for $1-$cycle]\label{1cycleallsol} Let $A$ be a cyclic matrix the same dimension $n$, given in the form \eqref{diagA}. If $X$ is a doubly stochastic matrix of order $n$, the following statements are equivalent:
		\begin{itemize}
			\item[(a)] $X$ is a d. s. solution to \eqref{YBME}.
			\item[(b)] There exist numbers $d,n_0\in\{1,\ldots,n\}$ with $n=d\cdot n_0$, such that the matrix $X$ has the form
			\begin{equation}
				\label{Kronecker}
				X=U_{n_0}\otimes D,
			\end{equation}
			where $\otimes$ stands for the Kronecker product, $U_{n_0}$ is the full $n_0\times n_0$ uniform distribution matrix, and $D\in\Omega_d$ is arbitrary. Consequently, $X\in\U_n(n_0)$.
		\end{itemize} 
	\end{theorem}
	\begin{proof} 
		Assume that $X$ is a permutation matrix. Denote by 
		$$\operatorname{Supp}X=\{(i,j): X(i,j)=1\}.$$
		Then  \eqref{Dperm} holds for $X$ if and only if for every $i,j$: \begin{equation}\label{support}(i,j)\in\operatorname{Supp}_X\Leftrightarrow \left(\exists s_{(i,j)}\right)\ (i+1,s_{(i,j)}),(s_{(i,j)}-1,j-1)\in\operatorname{Supp}_X.\end{equation}
		On the other hand, for an arbitrary pair $(i,j)\in \operatorname{Supp}_X$, one has  that
		\begin{eqnarray}\label{Prodij}
			\begin{aligned}&\left(\exists s_0\in\{1,\ldots,n\}\right)\quad (i+1,s_0),(s_0-1,j-1)\in\operatorname{Supp}_X\Leftrightarrow \\
				&\left(\exists s_0\in\{1,\ldots,n\}\right)\quad X(i+1,s_0)\cdot X(s_0-1,j-1)=1\Leftrightarrow\\
				&X(i,j)=\sum_{s=1}^n X(i+1,s)\cdot X(s-1,j-1)\Leftrightarrow \quad\left(\textrm{by \eqref{condtij}}\right)\\
				&AXA(i,j)=XAX(i,j).
		\end{aligned}\end{eqnarray}
		The latter equivalence holds true for any $(i,j)\in\operatorname{Supp}X$, therefore \eqref{Dperm} holds for $X$, if and only if \eqref{support} holds for every $i,j$, if and only if \eqref{Prodij} holds for any $i,j$.
		In this case, the decomposition \eqref{Kronecker} is true by taking $D:=X$, $d=n$, $n_0=1$, and $U_{n_0}=[1]_{1\times 1}$ is the $1\times 1$ matrix.
		
		$(a)\Rightarrow(b)$: Let $X$ be a proper d. s. solution to \eqref{YBME}. Then by Theorem \ref{udmus} there exists an $n_0$ such that $X\in \U_n(n_0)$. By Theorem \ref{divisors} we know that $n_0|n$ and we denote by $d$ the number $d=\frac{n}{n_0}$.  If $n=n_0$, then $U_{n_0}=U_n=X$, $D=[1]_{1\times 1}$  and the claim is true. For $1<n_0<n$ we proceed as follows:
		
		Let $I_j$ and $J_i$ be defined as in \eqref{Ij} and \eqref{Ji}, respectively, for any $i,j\in\{1,\ldots,n\}$. By Theorem \ref{equidistant} the entries $\frac{1}{n_0}$ are evenly distributed throughout each row and column of the solution $X$, with precisely $d-1$ zeros between two consecutive entries $\frac{1}{n_0}$ along the same row, resp. column, modulo $n$. Denote by $Y_d$  the upper-left block of the matrix $X$ of dimension $d$:
		$$Y_d(i,j):=X(i,j),\quad i,j\in\{1,\ldots,d\}.$$
		Since the nonzero entries $\frac{1}{n_0}$ are evenly distributed through every row and every column of the matrix $X$, with $d-1$ zeros between any two consecutive nonzero entries (modulo $n$), it follows that in each row and in each column of the matrix $Y$ there is exactly one entry with the value $\frac{1}{n_0}$ while the remaining entries in $Y$ are zeros. Moreover, assume that for some $(i_0,j_0)\in\{1,\ldots,d\}^2$ 
		$$Y(i_0,j_0)=X(i_0,j_0)=\frac{1}{n_0}.$$
		Then by Corollary \ref{samesets} 
		$$I_{(j_0-_n1)}-_n1=J_{i_0+_n1}.$$
		However, the set $I_{(j_0-_n1)}-_n1$ has exactly one element $s_0$ from the set $\{1,\ldots, d\}$ due to the equidistant distribution of $\frac{1}{n_0}$ through the column $j_0-_n1$. Similarly, by Theorem \ref{equidistant}, there is exactly one entry $s_1$ from the set $J_{i_0+_n1}$ which belongs to the set $\{1,\ldots,d\}$. By \eqref{condtij}, it follows that 
		$$Y(i_0,j_0)=X(i_0+_n1,s_1)\cdot X(s_0,j_0-_n1)+\ldots$$
		implying that $s_0=s_1-_n1$. Notice that, if $s_1=1$, then $s_0=n$, which is congruent with $d$ modulo $d$, while if $s_1\in\{2,\ldots, d\}$ then $s_0$ remains to be in $\{1,\ldots,d-1\}$. By choosing $D:=n_0 Y$ it is clear that $D\in\Omega_d$ and \eqref{Kronecker} holds by construction.

		$(b)\Rightarrow (a):$ Conversely, assume that $X$ is given as in \eqref{Kronecker}, for some $d,n_0$ $dn_0=n$, and $D\in\Omega_d$ arbitrary (but fixed). Since the permutation solutions have been characterized in the previous part of the proof, we assume that $n_0,d<n$. 
		
		By construction, the entire matrix $X$ comprises out of $n_0^2$ identical copies of the upper left block of dimension $d\times d$, therefore it has exactly one nonzero entry (with the value $\frac{1}{n_0}$) in every row and in every column in the said upper left $d\times d$ block.

		Let $(i-1,j+1)$ be an arbitrary nonzero entry, $X(i-1,j+1)=\frac{1}{n_0}$, where $1\leq i,j\leq d$. Then from \eqref{Dperm}, there exists a unique $\ell_{i,j}\in\{1,\ldots,d\}$ such that $$X(i+_d1,\ell_{i,j})=\frac{1}{n_0}=X(\ell_{i,j}-_d1,j-_d1),$$
		implying
		$$I_{j-_d1}=\{\ell_{i,j}-_d1,\ell_{i,j}-_d1+_n d, \ell_{i,j}-_d1+_n 2d,\ldots, \ell_{i,j}-_d1+_n (n_0-1)d\}$$
		and
		$$J_{i+_d1}=\{\ell_{i,j},\quad \ell_{i,j}+d,\quad\ldots,\ell_{i,j}+(n_0-1)+_nd\}.$$
		Then by Proposition \ref{YAcirc} it follows that
		\begin{eqnarray}
			\label{entrywisequal}
			\begin{aligned}
				AXA(i,j)&=X(i-_d1,j+_d 1)=\\
				&=\sum_{s=0}^{n_0-1}X(i+_d1, \ell_{i,j}+sd)\cdot X(\ell_{i,j}-_d1+sd,j-_d1)=\\
				&=\sum_{s=0}^n X(i+_d1,s)\cdot X(s-_n1,j-_d1)=XAX(i,j).
		\end{aligned}\end{eqnarray}
		Applying this procedure for any $d\times d$ sub-block of the matrix $X$, we verify that all nonzero entries of the matrix $X$ (all $n_0^2$ of them) satisfy \eqref{entrywisequal}. Since the remaining entries of $X$ are zeros, direct verification shows that for every $1\leq i,j\leq n$ the condition \eqref{condtij} holds, thus proving that $X$ is indeed a doubly stochastic solution to \eqref{YBME}.
	\end{proof}
	Within the obtained solution set, we characterize the commuting solutions in terms of the following:
	\begin{lemma}\label{lem:commutant_cyclic_is_circulant}
		Let $A$ be the cyclic permutation matrix of dimension $n$ given by \eqref{diagA}.
		If $X\in\mathbb{C}^{n\times n}$ satisfies $XA=AX$, then $X$ is a circulant matrix.
	\end{lemma}
	
	\begin{proof}
		Let $\omega:=e^{2\pi i/n}$ and let $F\in\mathbb{C}^{n\times n}$ be the unitary Fourier matrix
		\[
		F=\frac1{\sqrt n}\big(\omega^{(j-1)(k-1)}\big)_{j,k=1}^n .
		\]
		By [10, Chapter~3, Theorem~3.1], the matrix $A$ is diagonalized by $F$, namely
		\[
		F^*AF=\Lambda:=\mathrm{diag}(1,\omega,\omega^2,\dots,\omega^{n-1}),
		\]
		whose diagonal entries are pairwise distinct.
		
		Assume $XA=AX$ and set $B:=F^*XF$. Multiplying $XA=AX$ on the left by $F^*$ and on the right by $F$
		gives
		\[
		B\Lambda=\Lambda B.
		\]
		Since $\Lambda$ is diagonal, taking the $(i,j)$--entries yields
		\[
		B_{ij}\Lambda_{jj}=\Lambda_{ii}B_{ij},
		\quad\text{hence}\quad
		B_{ij}(\Lambda_{jj}-\Lambda_{ii})=0.
		\]
		If $i\neq j$, then $\Lambda_{jj}\neq\Lambda_{ii}$, so $B_{ij}=0$. Therefore $B$ is diagonal and
		\[
		X=FBF^*.
		\]
		By \cite[Chapter $3$, Theorem $3.1$]{RMG}, matrices of the form $FBF^*$ with $B$ diagonal are precisely the
		circulant matrices. Hence $X$ is circulant.
	\end{proof}
	
	\begin{theorem}[Commuting d. s. solutions for the $1-$cycle]\label{commuting_solutions_1cycle}
		Assume that $A$ is the cyclic permutation matrix of order $n$ given by \eqref{diagA}.
		Let $X$ be a doubly stochastic solution to \eqref{YBME}, given by \eqref{Kronecker} as 
		$$X=U_{n_0}\otimes D,$$
		for some $D\in\Omega_d$, $n_0\cdot d=n$. 
		
		Then $XA=AX$ if and only if
		the matrix $D$ is circulant.
	\end{theorem}
	
	\begin{proof}
		($\Rightarrow$)
		Assume $XA=AX$. By Lemma~\ref{lem:commutant_cyclic_is_circulant}, the matrix $X$ is circulant, hence
		\[
		X(i+1,j+1)=X(i,j)
		\qquad (\mathrm{mod}\ n).
		\]
		By Theorem \ref{1cycleallsol}, $X=U_{n_0}\otimes D$ with $n=n_0\cdot d$ and $D\in\Omega_d$.
		Write $i=\alpha d+r$ and $j=\beta d+s$ with $\alpha,\beta\in\{0,\dots,n_0-1\}$ and $r,s\in\{1,\dots,d\}$.
		Then
		\[
		X(i,j)=\frac1{n_0}D(r,s).
		\]
		Comparing $X(i+1,j+1)$ with $X(i,j)$ gives
		\[
		D(r+1,s+1)=D(r,s)
		\qquad (\mathrm{mod}\ d),
		\]
		with indices taken modulo $d$. Thus $D$ is circulant.
		
		($\Leftarrow$)
		Conversely, if $D$ is circulant, then $D(r+1,s+1)=D(r,s)$ (mod $d$), and the same block computation
		implies $X(i+1,j+1)=X(i,j)$ (mod $n$). Hence $X$ is circulant. Since $A$ is circulant as well, $XA=AX$.
	\end{proof}
	\begin{corollary}\label{trivA}
		A permutation solution $X$ is commuting if and only if $X=A$.
	\end{corollary}
	This observation follows immediately from $AX^2=A^2X\Leftrightarrow A=X$, if both $A$ and $X$ are invertible. Corollary \ref{trivA} is just a consequence of Theorem \ref{commuting_solutions_1cycle} which confirms this well-known fact.
	
	\begin{example}
		Let $n=8$ and suppose $A$ is in the form \eqref{diagA}:
		$$A=\left[\begin{array}{cccccccc}
			0 & 0 & 0& 0 & 0 & 0 & 0 & 1\\
			1 & 0 & 0 & 0 & 0 & 0 & 0 & 0\\
			0 & 1 & 0 & 0 & 0 & 0 & 0 & 0\\
			0 & 0 & 1 & 0 & 0 & 0 & 0 & 0\\
			0 & 0 & 0 & 1 & 0 & 0 & 0 & 0\\
			0 & 0 & 0 & 0 & 1 & 0 & 0 & 0\\
			0 & 0 & 0 & 0 & 0 & 1 & 0 & 0\\
			0 & 0 & 0 & 0 & 0 & 0 & 1 & 0\\
		\end{array}\right].$$
		Then $n_0\in\{1,2,4,8\}$ and $d=\frac{n}{n_0}$. Denote by $\pi_A=(1 \ 2 \ 3\ 4\ 5\ 6\ 7\ 8)$.\\
		
		\noindent\textbf{Case $1$: $n_0=1$.} When $n_0=1$, then $d=8$ and the corresponding d. s. solutions $X_8$ are permutation solutions, provided as $X_8=U_1\otimes D_8\equiv D_8$, where $D_8\in\Omega_8$ is a (permutation) matrix satisfying \eqref{Dperm}. All such solutions were obtained in Section \ref{Perm}, Example \ref{Ex_8}, and a direct verification shows that they all posses the property \eqref{Dperm}.\\

		\noindent\textbf{Case $2$: $n_0=2$.} When $n_0=2$ then $d=4$, and we have four possibilities for the circulant permutations $D_4$:
		$$D_{4,1}=\left[\begin{array}{cccc}
			1 & 0 & 0 & 0\\
			0 & 1 & 0 & 0\\
			0 & 0 & 1 & 0\\
			0 & 0 & 0 & 1
		\end{array}\right],\;D_{4,2}=\left[\begin{array}{cccc}
			0 & 0 & 1 & 0\\
			0 & 0 & 0 & 1\\
			1 & 0 & 0 & 0\\
			0 & 1 & 0 & 0
		\end{array}\right],$$
		$$D_{4,3}=\left[\begin{array}{cccc}
			0 & 1 & 0 & 0\\
			0 & 0 & 1 & 0\\
			0 & 0 & 0 & 1\\
			1 & 0 & 0 & 0
		\end{array}\right],\;
		D_{4,4}=\left[\begin{array}{cccc}
			0 & 0 & 0 & 1\\
			1 & 0 & 0 & 0\\
			0 & 1 & 0 & 0\\
			0 & 0 & 1 & 0
		\end{array}\right].$$
		Notice that $D_{4,1},D_{4,2}$, and $D_{4,3}$ do not belong to $\Omega_4$, while $D_{4,4}\in\Omega_4$. There are four additional possibilities for the non-circulant matrices $D_4$ which satisfy \eqref{Dperm}:
		$$D_{4,5}=\left[\begin{array}{cccc}
			0 & 0 & 0 & 1\\
			0 & 0 & 1 & 0\\
			1 & 0 & 0 & 0\\
			0 & 1 & 0 & 0
		\end{array}\right],\;D_{4,6}=\left[\begin{array}{cccc}
			0 & 0 & 1 & 0\\
			0 & 0 & 0 & 1\\
			0 & 1 & 0 & 0\\
			1 & 0 & 0 & 0
		\end{array}\right],$$
		$$D_{4,7}=\left[\begin{array}{cccc}
			0 & 1 & 0 & 0\\
			0 & 0 & 0 & 1\\
			1 & 0 & 0 & 0\\
			0 & 0 & 1 & 0
		\end{array}\right],\;D_{4,8}=\left[\begin{array}{cccc}
			0 & 0 & 1 & 0\\
			1 & 0 & 0 & 0\\
			0 & 0 & 0 & 1\\
			0 & 1 & 0 & 0
		\end{array}\right].$$
		Now we have
		$$U_2=\frac{1}{2}\left[\begin{array}{cc}
			1 & 1\\
			1 & 1
		\end{array}\right]$$
		and
		$$X_{4,k}=U_2\otimes D_{4,k},\quad k\in\{4,\ldots,8\}.$$
		There is only one commuting solution, $X_{4,4}$, while $X_{4,5},$ $\ldots,$ $X_{4,8}$ are non-commuting solutions. \\
		
		\noindent\textbf{Case $3$: $n_0=4$.} When $n_0=4$ then $d=2$, and we have circulant matrices $D_2$:
		$$D_{21}=\mat{1}{0}{0}{1},\quad D_{22}=\mat{0}{1}{1}{0}.$$
		Direct verification shows that $D_{21}\notin\Omega_2$ and $D_{22}\in\Omega_2$.
		Respectively, 
		$$U_4=\frac{1}{4}\left[\begin{array}{cccc}
			1 & 1 & 1 & 1\\
			1 & 1 & 1 & 1\\
			1 & 1 & 1 & 1\\
			1 & 1 & 1 & 1
		\end{array}\right]$$
		and we have one more commuting solution:
		$$X_{2,2}=U_4\otimes D_{2,2}.$$
		\noindent\textbf{Case $4$: $n_0=8$.} Finally when $n_0=8$ then $d=1$ and $D_1=[1]_{1\times 1}$, therefore $X_1=\frac{1}{8}\mathbf{1}_{8}=U_8$ is a proper d. s. solution to \eqref{YBME}, again a commuting one.\\
		
		In summary, there are in total $7$ proper doubly stochastic solutions, out of which $3$ are commuting solutions:
		$$U_8,\quad U_4\otimes\mat{0}{1}{1}{0},\quad U_2\otimes\left[\begin{array}{cccc}
			0 & 0 & 0 & 1\\
			1 & 0 & 0 & 0\\
			0 & 1 & 0 & 0\\
			0 & 0 & 1 & 0
		\end{array}\right],$$
		and four of which  are non-commuting solutions:
		$$U_2\otimes\left[\begin{array}{cccc}
			0 & 0 & 0 & 1\\
			0 & 0 & 1 & 0\\
			1 & 0 & 0 & 0\\
			0 & 1 & 0 & 0
		\end{array}\right],\;U_2\otimes\left[\begin{array}{cccc}
			0 & 0 & 1 & 0\\
			0 & 0 & 0 & 1\\
			0 & 1 & 0 & 0\\
			1 & 0 & 0 & 0
		\end{array}\right],$$
		$$U_2\otimes\left[\begin{array}{cccc}
			0 & 1 & 0 & 0\\
			0 & 0 & 0 & 1\\
			1 & 0 & 0 & 0\\
			0 & 0 & 1 & 0
		\end{array}\right],U_2\otimes\left[\begin{array}{cccc}
			0 & 0 & 1 & 0\\
			1 & 0 & 0 & 0\\
			0 & 0 & 0 & 1\\
			0 & 1 & 0 & 0
		\end{array}\right].$$
		Moreover, there are $25$ permutation solutions, with all being non-commuting except for the trivial one $\pi_X=\pi_A$. 
		\ \hfill$\clubsuit$
	\end{example}

	\section{Proper d. s. solutions-general case}\label{general}
	
	Finally, in this section we assume that $A$ is an arbitrary permutation matrix and we proceed to find all proper doubly stochastic solutions to \eqref{YBME}. 
	
	\subsection{Necessary conditions}
	
	Assume that $X$ is a proper d. s. solution to \eqref{YBME}. The fact that $X$ must be a singular matrix (by Lemma \ref{regsing}) leads  to the following observation:
	\begin{lemma}\label{propper}
		Let $X$ be a proper d. s. solution to \eqref{YBME}. Then $\N(X)$, $\R(X)$, and $\N(X)^\perp$ and $\R(X)^\perp$ are $A-$invariant and $X-$invariant subspaces.
	\end{lemma}
	\begin{proof}
		Since $AXA=XAX$ and $XAX:\N(X)\to \{0\}$, it follows that 
		$$AXA:\N(X)\to\{0\}.$$
		However, $A$ is invertible, therefore $AXA:\N(X)\to\{0\}\Leftrightarrow XA:\N(X)\to \{0\}\Leftrightarrow A:\N(X)\to \N(X)$. Similarly,
		$$\R(X)\supset\R(XAX)=\R(AXA)=\R(AX)=A\left(\R(X)\right).$$
		On the other hand, recall that every permutation is a unitary matrix (operator), so $W$ is an invariant subspace for $A$ if and only if $W^\perp$ is an invariant subspace for $A$: for any $u\in W^\perp$ and any $v\in W$ we have
		$$0=\langle u, Av\rangle=\langle A^*u,v\rangle,$$
		thus $W^\perp$ is $A^*-$invariant, and since $A^*$ is invertible the set equality $\left(A^*\right)^2(W^\perp)=A^*(W^\perp)=W^\perp$ is true. Then 
		$$0=\langle A^*u,v\rangle=\langle A^*AA^* u, v\rangle=\langle A\left(A^*\right)^2u, v\rangle=\langle \left(A^*\right)^2u, A^*v\rangle,$$
		implying that $A^*v\in W$, i.e. $W$ is $A^*-$invariant subset, and $A^*(W)=\left(A^*\right)^2(W)=W$, setwise speaking. This concludes that
		$$0=\langle u, A^*v\rangle=\langle Au,v\rangle$$
		so $W^\perp$ is also $A-$invariant subspace.
		
		To prove that $X:\N(X)^\perp\to \N(X)^\perp$, assume otherwise: suppose there exists a nonzero vector $v\in\N(X)^\perp$ such that $Xv\in \N(X)$. Then from the proved part of the claim it follows that
		$$0=X\underbrace{AXv}_{\in\N(X)}=AXAv\Leftrightarrow XAv=0\Leftrightarrow Av\in\N(X),$$
		which contradicts that $A:\N(X)^\perp\to\N(X)^\perp$.
		Finally, since $A$ is a bijection on both $\R(X)$ and $\R(X)^\perp$ we have
		$$\R(X^2)=X\left(\R(X)\right)=X\left(\R(AX)\right)=\R(XAX)=\R(AXA)=\R(X).$$
		So if there exists a $w\in\R(X)^\perp$ such that $Xw\in\R(X)=\R(X^2)$ then $w\in \R(X)$ ergo $w=0$.
	\end{proof}
	
	\begin{lemma}\label{partialisom} If $X$ is a proper d. s. solution to \eqref{YBME}, then  $X$ is a normal partial isometry.
	\end{lemma}
	\begin{proof}
		Lemma \ref{propper} implies that the matrix $A$ allows the decomposition:
		\begin{equation}
			\label{uniA}
			A=\mat{A_1}{0}{0}{A_4}:\left[\begin{array}{c}\N(X)\\\N(X)^\perp\end{array}\right]\to\left[\begin{array}{c}\N(X)\\\N(X)^\perp\end{array}\right],
		\end{equation}
		where the operator $A$ from \eqref{uniA} remains a unitary transform, and its eigenvalues remain intact (having the moduli equal to $1$). On the other hand, the operator $X$ now decomposes as
		\begin{equation}
			\label{uniX}
			X=\mat{0}{0}{0}{X_4}:\left[\begin{array}{c}\N(X)\\\N(X)^\perp\end{array}\right]\to\left[\begin{array}{c}\N(X)\\\N(X)^\perp\end{array}\right]. 
		\end{equation}
		Plugging in \eqref{uniA}--\eqref{uniX} into \eqref{YBME} gives
		\begin{equation}\label{unitaryX4}X_4A_4X_4=A_4X_4A_4.\end{equation}
		Moreover, notice that $\sigma(X_4)=\sigma(X)\setminus\{0\}$. The fact that $X$ is a partial isometry follows from
		$$\|Xu\|=\|AXAu\|=\|XAXu\|\leq \|AXu\|=\|Xu\|$$
		for any $u\in\N(X)^\perp$, so the inequality in the latter is actually an equality and the operator $X$ preserves the vector norm on $\N(X)^\perp$. By definition, $X$ is a partial isometry, while $X_4$ is an isometry in $\N(X)^\perp$. But the space $\N(X)^\perp$ is finite-dimensional, so $X_4$ is also surjective, meaning that it must be unitary (a surjective isometry), rendering the solution matrix $X$ from \eqref{uniX} normal by construction. 
	\end{proof} 
	
	Recall that a matrix $Y$ is said to be an inner inverse for $X$ if $XYX=X$. 
	
	\begin{corollary}
		Let $X$ be a proper doubly stochastic solution to \eqref{YBME}. There exists an inner inverse $Y$ for $X$ which is also a doubly stochastic matrix.
	\end{corollary}
	\begin{proof}
		From \cite[Theorem 3]{Prasad} it follows that any normal d. s. matrix is a partial isometry if and only if it has a doubly stochastic inner inverse. By Lemma \ref{partialisom} every proper d. s. solution is a normal partial isometry, so the claim is true.
	\end{proof}
	\noindent Recall the following results (see e. g. \cite{cruz}, \cite{cruzcor}, \cite{JLD},  \cite{HKF}, \cite{SGHSSP}, \cite{RJPGR}, and \cite{AWMIOBCA}):
	
	\begin{theorem} \label{pi} If $X$ is a d. s. partial isometry, then $X$ is the product of a permutation matrix and a doubly stochastic projection. 
	\end{theorem}
	
	\begin{theorem}\label{projection} Every doubly stochastic projection (idempotent) is permutationally similar to a block-diagonal matrix
			$$\left(\frac{1}{k_1}\mathbf{1}_{k_1}\right)\oplus\left(\frac{1}{k_2}\mathbf{1}_{k_2}\right)\oplus\ldots\oplus\left(\frac{1}{k_p}\mathbf{1}_{k_p}\right),$$
			for some positive integer $p$ and some positive integers $k_1,\ldots,k_p$.
		\end{theorem}
	
Let $Z$ be any doubly stochastic projector. There exists an arrangement of the base vectors $e_1,\ldots,e_n$ such that $Z$ reads in that basis as
	\begin{equation}
		\label{diagZ}Z=\operatorname{diag}\left(\frac{1}{k_1}\mathbf{1}_{k_1},\ldots,\frac{1}{k_p}\mathbf{1}_{k_p}\right)=\left[\begin{array}{ccccc}\frac{1}{k_1}\mathbf{1}_{k_1} & 0 & 0 &\cdots & 0\\
			0 & \frac{1}{k_2}\mathbf{1}_{k_2} & 0 & \cdots & 0\\
			\vdots & \vdots & \ddots & \vdots & \vdots\\
			0 & 0 & \ldots & 0 & \frac{1}{k_p}\mathbf{1}_{k_p}\end{array}\right],\end{equation}
	for some positive integer $p$ and some positive (not necessarily different) integers $k_1,\ldots,k_p$. Moreover, we assume the base vectors $e_1,\ldots,e_n$ to be arranged in such a fashion so that the same integers $k_i$ are grouped together, i.e., there exists an integer $m\in\{1,\ldots, p\}$ and there exist different integers $k_{p_1},\ldots, k_{p_m}$, such that the matrix $Z$ from \eqref{diagZ}  reads
		\begin{equation}
			\label{diagpZ}Z=\left[\begin{array}{ccccc}Z_{p_1} & 0 & 0 &\cdots & 0\\
				0 & Z_{p_2} & 0 & \cdots & 0\\
				\vdots & \vdots & \ddots & \vdots & \vdots\\
				0 & 0 & \ldots & 0 & Z_{p_{m}}\end{array}\right],\end{equation} 
	where 	\begin{equation}
			\label{Zpi}Z_{p_i}=\operatorname{diag}\left(\frac{1}{k_{p_{i}}}\mathbf{1}_{k_{p_i}},\ldots,\frac{1}{k_{p_{i}}}\mathbf{1}_{k_{p_i}}\right) 
		\end{equation}
	is a constant block-diagonal matrix corresponding to the integer $k_{p_i}$, for each $i\in\{1,\ldots,m\}$. Accordingly, we denote their dimensions as $d_{p_i}=\dim Z_{p_i}$ for each $i$. From \eqref{Zpi} it is clear that $k_{p_i}|d_{p_i}$, and we define $s_{p_i}=\frac{d_{p_i}}{k_{p_i}}$ for each $i\in\{1,\ldots,m\}$.

If  there is  a block $Z_{p_i}=U_{k_{p_i}}\equiv U_{d_{p_i}}$ which is the uniform distribution matrix (of conformable dimensions), then  we have $\R(Z_{p_i})=\operatorname{span}\{\mathbf{u}\}$, where $\mathbf{u}=(1,\ldots,1)^t\in \RR^{d_{p_i}}$ is the ones' vector, and $\N(Z_{p_i})=\R(Z_{p_i})^\perp$ in the space $\RR^{d_{p_i}}$. In that case $s_{p_i}=1=\dim \R(Z_{p_i})$.

More generally, if a block $Z_{p_i}$ consists out of  several blocks $U_{k_{p_i}}$, then, by \eqref{Zpi}, $Z_{p_i}$ comprises out of exactly $s_{p_i}$ identical blocks $U_{k_{p_i}}$, and by construction $Z_{p_i}\in U_{d_{p_i}}(k_{p_i})$. In that case, we have
				$$\R(Z_{p_i})=\operatorname{span}\{\mathbf{u}_j,\quad 1\leq j\leq s_{p_i}\},$$
				where
				$$\mathbf{u}_j=\left(\underbrace{0,\ldots,0}_{(j-1)\cdot k_{p_i}},\underbrace{1,\ldots,1}_{k_{p_i}},\underbrace{0,\ldots,0}_{(s_{p_i}-j)\cdot k_{p_i}}\right)^t,\quad 1\leq j\leq s_{p_i},$$
$\dim \R(Z_{p_i})=s_{p_i}$, and $\N(Z_{p_i})=\R(Z_{p_i})^\perp$. In other words, each block $Z_{p_i}$ is an orthogonal projector on the space $\RR^{d_{p_i}}$, and by \eqref{diagpZ}, so is the matrix $Z$. This gives the following  result:
	\begin{lemma}\label{necescomm}  Let $X$ be a proper doubly stochastic solution to \eqref{YBME}, represented as 
			\begin{equation}\label{XZQ}
				X=ZQ,
			\end{equation}
			where $Q$ is a permutation matrix and $Z$ is a doubly stochastic projection. 
		Then $XX^*=X^*X=Z$ and $$[Q,Z]=[X,Z]=[A,Z]=0.$$
	\end{lemma}
	\begin{proof} Since $X$ is a normal partial isometry and $Z$ is an orthogonal projector, we have
		$$Z=Z^2=ZQQ^*Z=ZQQ^*Z^*=XX^*=X^*X
		=Q^*Z^2Q=Q^*ZQ$$
		implying that $Z=XX^*=X^*X$, $QZ=ZQ$, and, consequently, $ZX=XZ$. Ergo $\N(X)=\N(Z)$. Since $\N(X)$ is $A-$invariant subspace, it follows that $A:\N(Z)\to\N(Z)$, and $(AZ)\upharpoonright_{\N(Z)}=(ZA)\upharpoonright_{\N(Z)}$. Moreover, from $AXA=XAX$ we get
		$$AXA=XX^*XAXX^*X=ZXAXZ=ZAXAZ,$$
		so the matrix $AXA$ satisfies
		$$ZAXA=Z^2AXAZ=AXA=ZAXAZ^2=AXAZ,$$
		i.e., $AXA$ commutes with $Z$ as well. But then $XAX$ commutes with $Z$, and
		$$XZAX=ZXAX=XAXZ=XAZX$$
		implying that
		$$X(AZ-ZA)X=0.$$
		Since $X$ is invertible in $\N(X)^\perp$, and $\N(X)^\perp$ is $X-$invariant subspace, the latter implies that (since $\N(Z)^\perp\equiv\N(X)^\perp$): 
		$(AZ)\upharpoonright_{\N(Z)^\perp}=(ZA)\upharpoonright_{\N(Z)^\perp}$, concluding that $AZ=ZA$ in the entire space.
	\end{proof}
	
Since $Q$ commutes with $Z$ and each block $Z_{p_i}$ in \eqref{diagpZ} is different, it follows that $Q$ has the block-diagonal form
	\begin{equation}
		\label{Qp}
		Q=\left[\begin{array}{cccc}Q_{p_1} &0 & \ldots &0\\
			0& Q_{p_2} &0  & 0\\
			\vdots & \vdots & \ddots & \vdots \\
			0 & 0 & \ldots & Q_{p_m}\end{array}\right],\quad [Q_{p_i},Z_{p_i}]=0.
	\end{equation}  
	Consequently, the solution $X$ from \eqref{XZQ} is also block-diagonal, 
	\begin{equation}
		\label{Xp}
		X=\left[\begin{array}{cccc}
			Z_{p_1}Q_{p_1} &0 & \ldots &0\\
			0& Z_{p_2}Q_{p_2} &0  & 0\\
			\vdots & \vdots & \ddots & \vdots \\
			0 & 0 & \ldots & Z_{p_m}Q_{p_m}\end{array}\right],\quad [Q_{p_i},Z_{p_i}]=0,
	\end{equation} 
	where $X_{p_i}:=Z_{p_i}Q_{p_i}$ is by definition in $\U_{d_{p_i}}(k_{p_i})$. Moreover, since $A$ also commutes with $Z$, we conclude that
	\begin{equation}
		\label{Ap}
		A=\left[\begin{array}{cccc}A_{p_1} &0 & \ldots &0\\
			0& A_{p_2} &0  & 0\\
			\vdots & \vdots & \ddots & \vdots \\
			0 & 0 & \ldots & A_{p_m}\end{array}\right], \quad [A_{p_i},Z_{p_i}]=0,
	\end{equation}
	in that basis. We summarize the previous analysis in the following theorem:
	\begin{theorem}[Necessary conditions-general case]\label{NecesQ} For a proper doubly stochastic solution $X$ decomposed as $X=ZQ$ via \eqref{XZQ}, there exists an arrangement of the basis vectors in which $Z$ has the form \eqref{diagpZ}, $Q$ has the form \eqref{Qp}, the solution $X$ has the form \eqref{Xp}, and $A$ has the form \eqref{Ap}. Consequently, the solution $X$ is a commuting one, if and only if for every $i\in\{1,\ldots,m\}$ the matrix $Q_{p_i}$ commutes with $A_{p_i}$.
	\end{theorem}

	\subsection{Solving the block equations}
	
	Keeping the notation from the previous subsection, in what follows we proceed to characterize all doubly stochastic solutions. Recall that the operator $A$ is always unitary, and every $Z$ from \eqref{XZQ} must be an orthogonal projector. This means that $\R(Z)\perp\N(Z)$, and their direct (orthogonal) sum exhausts the entire space $\RR^n$. Moreover, $\R(Z)$ is $A-$invariant if and only if $\N(Z)$ is $A-$invariant subspace. Also recall the following well-known result:

\begin{lemma} Let $M,N$ be given square matrices, and $N=N^2$. Then $MN=NM$ if and only if $\R(N)$ and $\N(N)$ are $M-$invariant subspaces.
\end{lemma}
So $A$ commutes with $Z$ if and only if $\R(Z)$ and $\N(Z)$ are $A-$invariant. In fact, if $A$ has the form \eqref{Ap}, then the same is true for every block permutation $A_{p_i}$ and every block $Z_{p_i}$. So if $Z$ has the form \eqref{diagpZ} in some basis, then in that same basis $A$ (which commutes with $Z$) reads \eqref{Ap}. Here it is imperative to assume that different blocks $Z_{p_i}$ correspond to different integers $k_{p_i}$.

\begin{theorem}\label{equivQ} Let $A$ be a permutation matrix given as \eqref{Ap}, and let $Z$ be a doubly stochastic projector given as \eqref{diagpZ} which in addition commutes with $A$. For an arbitrary doubly stochastic matrix $X$, the following statements are equivalent:
	\begin{itemize}
		\item[(a)] The matrix $X$ is a solution to \eqref{YBME} and satisfies $\N(X)=\N(Z)$.
		\item[(b)] There exist permutation matrices $Q_{p_1},\ldots,Q_{p_m}$ of conformable dimensions, where each $Q_{p_i}$ commutes with the corresponding $Z_{p_i}$ and $Q_{p_i}$ solves the Yang-Baxter-like matrix equation
		\begin{equation}\label{A4Q4perm}\left(A_{p_i}Q_{p_i}A_{p_i}\right)\upharpoonright_{\R(Z_{p_i})}=\left(Q_{p_i}A_{p_i}Q_{p_i}\right)\upharpoonright_{\R(Z_{p_i})}.\end{equation}
	\end{itemize}
	If $(a)$ or $(b)$ hold, then with respect to \eqref{diagpZ} the matrix $X$ allows the decomposition \eqref{Xp}, that is,  $X=\operatorname{diag}\left(X_{p_1},\ldots,X_{p_m}\right)$ where $X_{p_i}=Z_{p_i}Q_{p_i}$, $1\leq i\leq m$.
	\end{theorem}
	
\begin{proof} $(a)\Rightarrow(b):$
If $X$ is a solution, then we automatically have that $X$ has the form \eqref{Xp}, assuming that $A$ and $Z$ are given as \eqref{Ap} and \eqref{diagpZ}, respectively. That is, due to \eqref{XZQ}, there exists a permutation matrix $Q$ which commutes with $Z$, and $X=ZQ$, so $\N(Z)=\N(X)$ by construction. Finally, we have that $Q$ has the form (again, assuming the basis in which $Z$ has the form \eqref{diagpZ} and $A$ has the form \eqref{Ap}) \eqref{Qp}, and consequently $X_{p_i}=Z_{p_i}Q_{p_i}$, $1\leq\ i\leq m$. But then, since $X_{pi}$ must solve the block YBME $A_{p_i}X_{p_i}A_{p_i}=X_{p_i}A_{p_i}X_{p_i}$, we have that
$$\begin{aligned}
&(A_{p_i}Q_{p_i}A_{p_i})\upharpoonright_{\R(Z_{p_i})}=(Z_{p_i}A_{p_i}Q_{p_i}A_{p_i})\upharpoonright_{\R(Z_i)}=(A_{p_i}X_{p_i}A_{p_i})\upharpoonright_{\R(Z_i)}=\\
&(X_{p_i}A_{p_i}X_{p_i})\upharpoonright_{\R(Z_{p_i})}=(Z_{p_i}Q_{p_i}A_{p_i}Q_{p_i}Z_{p_i})\upharpoonright_{\R(Z_{p_i})}=(Q_{p_i}A_{p_i}Q_{p_i})\upharpoonright_{\R(Z_{p_i})}.
\end{aligned}$$
$(b)\Rightarrow(a):$ Conversely, assume $A$ and $Q$ are permutation matrices which commute with $Z$. Then $A_{p_i}$ and $Q_{p_i}$ commute with $Z_{p_i}$, for each block $1\leq i\leq m$. Define $X_{p_i}:=Z_{p_i}Q_{p_i}$. Then $\N(X)=\N(Z)$ by construction, and 
$$\begin{aligned}
&(A_{p_i}X_{p_i}A_{p_i})\upharpoonright_{\N(Z_{p_i})}=(A_{p_i}Q_{p_i}A_{p_i}Z_{p_i})\upharpoonright_{\N(Z_{p_i})}=0=\\
&(Q_{p_i}A_{p_i}Q_{p_i}Z_{p_i})\upharpoonright_{\N(Z_{p_i})}=(Z_{p_i}Q_{p_i}A_{p_i}Q_{p_i}Z_{p_i})\upharpoonright_{\N(Z_{p_i})}=(X_{p_i}A_{p_i}X_{p_i})\upharpoonright_{\N(Z_{p_i})}.
\end{aligned}$$
On the other hand, $(Z_{p_i})\upharpoonright_{\R(Z_{p_i})}=I\upharpoonright_{\R(Z_{p_i})}$, so \eqref{A4Q4perm} implies that
$$\begin{aligned}
	&(X_{p_i}A_{p_i}X_{p_i})\upharpoonright_{\R(Z_{p_i})}=(Q_{p_i}A_{p_i}Q_{p_i})\upharpoonright_{\R(Z_{p_i})}=\\
	=&(A_{p_i}Q_{p_i}A_{p_i})\upharpoonright_{\R(Z_{p_i})}=(A_{p_i}X_{p_i}A_{p_i})\upharpoonright_{\R(Z_{p_i})},
\end{aligned}$$
so $X_{p_i}$ solves the block-equation in $\N(Z_{p_i})\oplus\R(Z_{p_i})=\RR^{d_{p_i}}$, that is, $X_{p_i}$ solves the equation in the entire space $\RR^{d_{p_i}}$. Consequently, their sum $X=\bigoplus\limits_{i=1}^mX_{p_i}$ solves the equation \eqref{YBME} in the initial space $\RR^n$. 
\end{proof}

\begin{remark}\label{rem1} If there exists a block $A_{p_i}$ for which one chooses $k_{p_i}=1$, then $Z_{p_i}=I_{d_{p_i}}$ is the identity matrix, and the block $X_{p_i}$ coincides with the permutation matrix $Q_{p_i}$. In that case, the space $\N(X_{p_i})=\N(Z_{p_i})=\N(I_{d_{p_i}})=\{0\}$ is trivial, and the block $X_{p_i}\equiv Q_{p_i}$ must solve the permutation equation \eqref{A4Q4perm} in the entire space $\RR^{d_{p_i}}=\N(Z_{p_i})^\perp$. These are \emph{permutation blocks} in the (generally speaking) proper doubly stochastic solution $X$, and they are characterized in terms of Section \ref{Perm}. Consequently, a doubly stochastic solution $X$ is a permutation solution if and only if all its blocks $X_{p_i}$ in \eqref{Xp} are permutation blocks, equivalently, if and only if all doubly stochastic projections $Z_{p_i}$ are invertible, $Z_{p_i}=I_{d_{p_i}}$, for every $i\in\{1,\ldots,m\}$. In that sense, all permutation solutions are contained within this procedure.
\end{remark}

\begin{remark} \label{rem2} Similarly, if  there is  a block $A_{p_i}$ for which one chooses $k_{p_i}=d_{p_i}$, then $Z_{p_i}=X_{p_i}=U_{d_{p_i}}$ is the uniform distribution matrix. In that case, we have $\R(X_{p_i})=\R(Z_{p_i})=\operatorname{span}\{\mathbf{u}\}$, where $\mathbf{u}=(1,\ldots,1)^t\in \RR^{d_{p_i}}$ is the ones' vector, and any permutation matrix $Q_{p_i}$ leaves the space $\operatorname{span}\{\mathbf{u}\}$ invariant,
	mapping the vector $\mathbf{u}$ into itself. This implies that $Q_{p_i}$ leaves the space $\N(Z_{p_i})$ invariant as well, and the ambiguity of $Q_{p_i}$ in $\N(Z_{p_i})$ is justified in this case: for any permutation $Q_{p_i}$ it holds that
	$$X_{p_i}=Q_{p_i}Z_{p_i}=Q_{p_i}U_{d_{p_i}}=U_{d_{p_i}}.$$ 
	So the permutation matrix $Q_{p_i}$ can indeed be arbitrarily chosen in $\N(Z_{p_i})$, because it automatically solves the block equation \eqref{A4Q4perm} in $\N(Z_{p_i})^\perp$.
\end{remark}

\begin{remark} More generally, if $d_{p_i}>k_{p_i}>1$, then by \eqref{Zpi} the block $Z_{p_i}$ consists out of several blocks $U_{k_{p_i}}$,
	and there exists a unique natural number $s_{p_i}$ such that $s_{p_i}k_{p_i}=d_{p_i}$. Then $Z_{p_i}$ comprises out of $s_{p_i}$ identical blocks $U_{k_{p_i}}$. In that case, we have
	$$\R(Z_{p_i})=\operatorname{span}\{\mathbf{u}_j,\quad 1\leq j\leq s_{p_i}\}$$
	where
	$$\mathbf{u}_j=\left(\underbrace{0,\ldots,0}_{(j-1)\cdot k_{p_i}},\underbrace{1,\ldots,1}_{k_{p_i}},\underbrace{0,\ldots,0}_{(s_{p_i}-j)\cdot k_{p_i}}\right)^t,\quad 1\leq j\leq s_{p_i}.$$
	It is clear that, in general, arbitrary permutation $L_{p_i}$ acting in the space $\RR^{d_{p_i}}$ would not necessarily leave the space $\R(Z_{p_i})$ invariant. However, Lemma \ref{necescomm} guarantees that the matrix $A$ commutes with $Z$, so $A_{p_i}$ must commute with $Z_{p_i}$, rendering the space $\R(Z_{p_i})$ $A_{p_i}-$invariant. Similarly, the unknown matrix $Q_{p_i}$ must leave the space $\R(Z_{p_i})$ invariant as well, due to the same Lemma. In order to solve \eqref{A4Q4perm} in $\R(Z_{p_i})$, one can use the results from Section \ref{Perm}, or some other papers, like \cite{{ChenYong}}, \cite{QHMSJDLZ}, or \cite{MSJDQHLZ}. However, once the restrictions of $Q_{p_i}$ to $\R(Z_{p_i})$ are found, one has to recover the initial permutation matrices $Q_{p_i}$. This procedure does not provide a unique permutation, but rather an entire family of permutations $Q_{p_i}$. 
\end{remark}
The main problem is now reduced to finding suitable doubly stochastic projectors $Z$, and the convenient basis rearrangements in which the sought DS projectors would have the form \eqref{diagpZ}. So, we can start with the followiong:
\begin{center}{\textbf{Procedure for DS solutions with prescribed kernel}}
		\end{center}
		\begin{itemize}
			\item[(Step 1.)] \textbf{Input}: Permutation matrix $A$.
			\item[(Step 2.)] Choose an arbitrary $A-$invariant subspace $W$ of $\RR^n$.
			\item[(Step 3.)] Denote by $Z_W$ the orthogonal projector for which $\N(Z_W)=W$. If $Z$ is not doubly stochastic, return to Step 2 and choose the space $W$ differently. Otherwise proceed to the following step.
			\item[(Step 4.)] For the (doubly stochastic) orthogonal projector $Z_W$ let $T_W$ be the basis rearrangement matrix such that $Z'=T_W^{-1}Z_WT_W$ has the form \eqref{diagpZ}.
			\item[(Step 5.)] Denote by $A':=T^{-1}_WAT_W$. Then $A'$ has the form \eqref{Ap}.
			\item[(Step 6.)] For each block $A'_{p_i}$ of the matrix $A'$, determine $Q'_{p_i}$ which is a permutation matrix that commutes with $Z'_{p_i}$ and solves \eqref{A4Q4perm}, where now we write $A'_{p_i}$ instead of $A_{p_i}$. 
			\item[(Step 7.)] Define $X'_{p_i}=Z_{p_i}Q'_{p_i}$,\quad  $X'=\bigoplus\limits_{i=1}^m X'_{p_i}$.
			\item[(Step 8.)] Define $X_W:=T_W X' T^{-1}_W$. 
		\end{itemize}
		The above procedure produces all DS solutions with the kernel $W$.   

	\begin{theorem}[All d. s. solutions]
		Let $A$ be a permutation matrix and let $X$ be a doubly stochastic matrix. Then $X$ is a solution to \eqref{YBME} if and only if there is an $A-$invariant subspace $W$ for which $X$ is obtained as $X\equiv X_W$ by the procedure described above. In that case, $\N(X)=W$.
	\end{theorem}
	\begin{proof} $(\Rightarrow):$ Follows from Theorem \ref{NecesQ} and Theorem \ref{equivQ}.
		
	$(\Leftarrow):$ By choosing $W$ as an $A-$invariant subspace, it follows that $W^\perp$ is also $A-$invariant. That being said, there exist a unique orthogonal projector $Z_W$ with $\N(Z_W)=W$ (equivalently, $\R(Z_W)=W^\perp$). If the projector $Z_W$ is not doubly stochastic then $W$ needs to be chosen differently. Otherwise, we can use Theorem \ref{projection} and rewrite $Z_W\equiv Z'$ so that it has the form \eqref{diagpZ}. To achieve this we use the basis rearrangement matrix $T_W$, which too must be a permutation matrix. In that case, one rewrites $A$ as $A'$ to obtain the form \eqref{Ap}. Now that By Theorem \ref{equivQ}, a matrix $X'$ is a solution to $A'X'A'=X'A'X'$ if and only if there exist suitably chosen permutation matrices $Q'_{p_i}$ which solve \eqref{A4Q4perm} with respect to $Z'_{p_i}$ and $A'_{p_i}$. Finally, by employing $T^{-1}_W$ we return the basis to the initial one, and by \eqref{similarity} we conclude that $X_W$ is a doubly stochastic solution to \eqref{YBME}.
		
	\end{proof}

	\begin{example}
	Let $\pi_A=(1\ 2\ 3\ 4)(5\ 6)$. Denote by  $A_1=(1\ 2\ 3\ 4)$, $A_2=(5\ 6)$. We identify the invariant subspaces for $A_1$ and $A_2$, since invariant subspaces of $A$ will comprise out of them. 
	\begin{itemize}
	\item Note that $A_1$ is a $4-$cycle, so we identify those invariant subspaces $W^{A_1}$ for $A_1$, which are kernels for the corresponding doubly stochastic orthogonal projectors $Z_{W^{A_1}}$ from $\RR^4$ to $\left(W^{A_1}\right)^\perp$:
	$$W^{A_1}_0=\{0\},\quad Z^{A_1}_0=I_4.$$
	$$W^{A_1}_1=\operatorname{span}\{(1,1,1,1)^t\}^\perp,\quad Z^{A_1}_1=U_4.$$
	If one chooses $W^{A_1}_{1'}=\left(W^{A_1}_{1}\right)^\perp$, then the corresponding orthogonal projector $Z^{A_1}_{1'}$ reads $$Z^{A_1}_{1'}=I_4-Z^{A_1}_1=I_4-U_4=\frac{1}{4}\left[\begin{array}{rrrr}3 & -1 & -1 & -1 \\
		-1 & 3 & -1 & -1\\-1 & -1 & 3 & -1 \\-1 & -1 & -1 & 3\end{array}\right],$$
		which is clearly not a doubly stochastic matrix. This is why we choose $Z^{A_1}_1$ instead of $Z^{A_1}_{1'}$, as explained in Step $3$ of the Procedure. Analogously, we have  $$W^{A_1}_2=\operatorname{span}\{(1, 0, -1, 0)^t, (0, 1, 0, -1)^t\},\quad Z^{A_1}_{2}=\frac{1}{2}\left[\begin{array}{cccc}1 & 0 & 1 & 0 \\
	 	0 & 1 & 0 & 1\\1 & 0 & 1 & 0 \\0 & 1 & 0 & 1\end{array}\right],$$
	 	where we choose $W^{A_1}_2$ instead of its orthogonal complement for the same reason: the projector $Z^{A_1}_{2'}=I_4-Z^{A_1}_2$ is not a doubly stochastic matrix.

	 	Notice that there are more $A_1$ invariant subspaces, 
	 	$$W_3=\operatorname{span}\{(1,-1,1,-1)^t\},\quad W_3^\perp,$$
	 	however, the corresponding orthogonal projector $Z^{A_1}_3$ with $W^{A_1}_3$ as its kernel reads
	 	$$ Z^{A_1}_3=\frac{1}{4}\left[\begin{array}{rrrr}3 & 1 & -1 & 1 \\
	 		1 & 3 & 1 & -1\\-1 & 1 & 3 & 1 \\1 & -1 & 1 & 3\end{array}\right],$$
	 		and, if one chooses $W_3^\perp$ instead of $W_3$, then the corresponding orthogonal projector reads
	 		$$Z^{A_1}_{3'}=
	 		\frac{1}{4}\left[\begin{array}{rrrr}1 & -1 & 1 & -1 \\
	 			-1 & 1 & -1 & 1\\1 & -1 & 1 & -1 \\-1 & 1 & -1 & 1\end{array}\right].$$
	 			Clearly $Z^{A_1}_3$ and $Z^{A_1}_{3'}$ are not doubly stochastic matrices, therefore we exclude them from our analysis and proceed with $W^{A_1}_0$, $W^{A_1}_1$ and $W^{A_1}_2$. In the cases for $W^{A_1}_0$ and $W^{A_1}_1$ the corresponding projectors are already in the form \eqref{diagpZ}, so there are no transformation matrices $T$ in these scenarios. For $W_0^{A_1}=\{0\}$ the set of permutation matrices $Q^{A_1}_0$  which solve $A_1Q^{A_1}_0A_1=Q^{A_1}_0A_1Q^{A_1}_0$, is precisely the set $\Omega_4$ from \eqref{Dperm}, by Remark \ref{rem1}. For $W^{A_1}_1$ the set of permutation matrices $Q^{A_1}_1$ which solve \eqref{A4Q4perm} in $\R(Z^{A_1}_1)=\left(W^{A_1}_1\right)^\perp$ is the set $ S_4$ due to Remark \ref{rem2}. Since $X^{A_1}=Z^{A_1}Q^{A_1}$ we have immediate solutions: 
	 			\begin{equation}
	 				\label{A101}
	 				X^{A_1}_1=U_4,\quad X^{A_1}_0\in\Omega_4-\textrm{arbitrary}.
		\end{equation}
		For the space $W^{A_1}_2$ the projector $Z^{A_1}_2$ is not in the form \eqref{diagpZ}. However, by introducing the matrix
		\begin{equation}
			\label{}
			T^{A_1}_2=\left(T^{A_1}_2\right)^{-1}=\left[\begin{array}{cccc}
			1 & 0 & 0 & 0\\0 & 0 & 1 & 0\\ 0 & 1 & 0 & 0\\0 & 0 & 0 & 1			\end{array}\right]
		\end{equation}
		we obtain $Z'^{A_1}_2=\mat{U_2}{0}{0}{U_2}$, while $A'_1$ reads
		$$A'_1=T^{A_1}_2A_1T^{A_1}_2=\left[\begin{array}{cccc}
			0 & 0 & 0 & 1\\0 & 0 & 1 & 0\\ 1 & 0 & 0 & 0\\0 & 1 & 0 & 0			\end{array}\right].$$
			The space $\R(Z'^{A_1}_2)$ is two-dimensional,
			$$\R(Z'^{A_1}_2)=\operatorname{span}\{(1,1,0,0)^t,(0,0,1,1)^t\},$$
			and the reduction of $A'_1$ to $\R(Z'^{A_1}_2)$ reads
			$$\left(A'_1\right)\upharpoonright_{\R\left(Z'^{A_1}_2\right)}=\mat{0}{1}{1}{0}.$$
			The only $(Q'^{A_1}_2)\upharpoonright_{\R\left(Z'^{A_1}_2\right)}$ which solves \eqref{A4Q4perm} is precisely $\left(A'_1\right)\upharpoonright_{\R\left(Z'^{A_1}_2\right)}$, so we have four choices for $Q'^{A_1}_2$:
		$$\begin{aligned}
			Q'^{A_1}_2\in & \left\{\left[\begin{array}{cccc}
	0 & 0 & 0 & 1\\0 & 0 & 1 & 0\\ 1 & 0 & 0 & 0\\0 & 1 & 0 & 0			\end{array}\right],\ \left[\begin{array}{cccc}
				0 & 0 & 1 & 0\\0 & 0 & 0 & 1\\ 1 & 0 & 0 & 0\\0 & 1 & 0 & 0			\end{array}\right],\right.\\
								& \left. \left[\begin{array}{cccc}
	0 & 0 & 0 & 1\\0 & 0 & 1 & 0\\ 0 & 1 & 0 & 0\\1 & 0 & 0 & 0	\end{array}\right],\ \left[\begin{array}{cccc}	0 & 0 & 1 & 0\\0 & 0 & 0 & 1\\ 0 & 1 & 0 & 0\\1 & 0 & 0 & 0			\end{array}\right]\right\}.
				\end{aligned}$$
Since $X'^{A_1}_2=Z'^{A_1}_2Q'^{A_1}_2$ and $X^{A_1}_2=T^{A_1}_2X'^{A_1}_2T^{A_1}_2$, we get for every $Q'^{A_1}_2$ calculated above:
\begin{equation}\label{XA12}
	X^{A_1}_2= \left[\begin{array}{cccc}
		0 & 1/2 & 0 & 1/2\\1/2 & 0 & 1/2 & 0\\ 0 & 1/2 & 0 & 1/2\\1/2 & 0 & 1/2 & 0			\end{array}\right].
		\end{equation}
		Define
		
\begin{equation}
\label{XA1}
X^{A_1}=\left\{U_4, \ P_4, \ \left[\begin{array}{cccc}
	0 & 1/2 & 0 & 1/2\\1/2 & 0 & 1/2 & 0\\ 0 & 1/2 & 0 & 1/2\\1/2 & 0 & 1/2 & 0			\end{array}\right]: P_4\in\Omega_4.\right\}.
\end{equation}
		\ \\\
		\ \\
\item For the block $A_2$ the invariant subspaces which allow doubly stochastic orthogonal projectors are just the spaces $W^{A_2}_1=\operatorname{span}\{(1,1)^t\}$ and $W^{A_2}_0=\{0\}$. By Theorem \ref{constr_noncom_0} the only permutation solution in this case is $X^{A_2}_0=A_2$, and in the case for $W^{A_2}_1$ we have $X^{A_2}_1=U_2$.
Combined, we write
\begin{equation}\label{XA2}
 		X^{A_{2}}=\left\{U_2,\quad \mat{0}{1}{1}{0}\right\}.
 		\end{equation}
 		\end{itemize}
 		The expressions \eqref{XA1} and \eqref{XA2} give one family of solutions,
 		\begin{equation}\label{X12}
 		\overline{X_{12}}=\left\{\mat{X_1}{0}{0}{X_2}: X_1\in X^{A_1}, X_2\in X^{A_2}\right\},\end{equation}
 		but these are not all solutions. 
 		\begin{itemize}
 			\item Indeed, another class emerges when we combine the invariant subspaces of $A$ which have the same integer $k$. For example, if we consider the space
 		$$W_6=W^{A_1}_2\oplus W^{A_2}_1$$
 		then the corresponding orthogonal projector $Z_6$ is doubly stochastic, $Z_6=Z^{A_1}_2\oplus U_2$, and in order to fit it into the form \eqref{diagpZ}, we use $T^{A_1}_{2}$ to obtain $T_6=T^{A_1}_2\oplus I_2$ and
 		$$Z'_6=T_6Z_6T_6=U_2\oplus U_2\oplus U_2.$$ Now $A'_6=A'_1\oplus A_2$, and 
 		$$A'_6=\left[\begin{array}{cccccc}
 			0 & 0 & 0 & 1 & 0 & 0\\0 & 0 & 1 & 0& 0 & 0\\ 1 & 0 & 0 & 0 &0 & 0\\0 & 1 & 0 & 0& 0 & 0\\ 0 & 0 & 0 &0 & 0 & 1\\
 			0 & 0 & 0 &0 & 1 & 0	\end{array}\right].$$
 			Restricted to $\R(Z'_6)$, $A'_6$ reads
 			$$\left(A'_6\right)\upharpoonright_{\R(Z'_6)}=\left[\begin{array}{ccc}
 				0 & 1 & 0\\
 				1 & 0 & 0 \\
 				0& 0 & 1\end{array}\right].$$
 				The corresponding permutation $\pi_6$ reads $\pi_6=(3)(1\ 2)$. So the solution permutation $\pi_{Q'_6}$ must have one fixed point and one cycle of length $2$. The choices are
 				$$\pi_{Q'_6}\in\{\pi_6,\ (2)(1\ 3), \ (1) (2\ 3)\}$$
 				and
 				$$\left(Q'_6\right)\upharpoonright_{\R(Z'_6)}\in\left\{\left[\begin{array}{ccc}
 					0 & 1 & 0\\
 					1 & 0 & 0 \\
 					0& 0 & 1\end{array}\right],\ \left[\begin{array}{ccc}
 					1 & 0 & 0\\
 					0 & 0 & 1 \\
 					0& 1 & 0\end{array}\right],\ \left[\begin{array}{ccc}
 					0 & 0 & 1\\
 					0 & 1 & 0 \\
 					1& 0 & 0\end{array}\right] \right\}.$$
 					This gives different choices for the matrix $Q'_6$, 
 					$$\begin{aligned}Q'_6\in&\left\{\left[\begin{array}{cccccc}
 						0 & 0 & a_1 & b_1 & 0 & 0\\0 & 0 & c_1 & d_1& 0 & 0\\ e_1 & f_1 & 0 & 0 &0 & 0\\g_1 & h_1 & 0 & 0& 0 & 0\\ 0 & 0 & 0 &0 & x_1 & y_1\\
 						0 & 0 & 0 & 0 & w_1 & z_1	\end{array}\right],\right.\\
 						&\left.\left[\begin{array}{cccccc}
 						a_2 & b_2 & 0 & 0 & 0 & 0\\c_2 & d_2 & 0 & 0& 0 & 0\\ 0 & 0 & 0 & 0 &e_2 & f_2\\0 & 0 & 0 & 0& g_2 & h_2\\ 0 & 0 & x_2 &y_2 & 0 & 0\\
 						0 & 0 & w_2 & z_2 & 0 & 0	\end{array}\right],\left[\begin{array}{cccccc}
 						0 & 0 & 0 & 0 & a_3 & b_3\\0 & 0 & 0 & 0& c_3 & d_3\\ 0 & 0 & e_3 & f_3 &0 & 0\\0 & 0 & g_3 & h_3& 0 & 0\\ x_3 & y_3 & 0 &0 & 0 & 0\\
 						w_3 & z_3 & 0 & 0 & 0 & 0	\end{array}\right]\right\},\end{aligned}$$
 						where $\mat{a_i}{b_i}{c_i}{d_i}$, $\mat{e_i}{f_i}{g_i}{h_i}$, and $\mat{x_i}{y_i}{w_i}{z_i}$ are arbitrary permutation blocks, independent from one another. This will give
 							$$\begin{aligned}
 								X'_6\in&\left\{\left[\begin{array}{cccccc}
 								0 & 0 & 1/2 & 1/2 & 0 & 0\\0 & 0 & 1/2 & 1/2& 0 & 0\\ 1/2 & 1/2 & 0 & 0 &0 & 0\\1/2 & 1/2 & 0 & 0& 0 & 0\\ 0 & 0 & 0 &0 & 1/2 & 1/2\\
 								0 & 0 & 0 &0  & 1/2 & 1/2	\end{array}\right],\right.\\
 								&\left.\left[\begin{array}{cccccc}
 								 1/2 & 1/2 & 0 & 0 &0 &0\\ 1/2 & 1/2& 0 & 0 & 0 & 0\\ 0 & 0 & 0 & 0 & 1/2 & 1/2\\0 &0 &0 & 0& 1/2 & 1/2\\ 0 & 0 & 1/2 & 1/2&0 &0 \\
 								0 & 0 & 1/2 & 1/2& 0 &0 	\end{array}\right],\left[\begin{array}{cccccc}
 							0 & 0 & 0 & 0 & 1/2 & 1/2\\0 & 0 & 0 & 0 & 1/2 & 1/2\\ 0 & 0 & 1/2 & 1/2 &0 & 0\\ 0 & 0 & 1/2 & 1/2 &0 & 0\\ 1/2 & 1/2 & 0 &0 & 0 & 0\\
 							1/2 & 1/2 & 0 & 0 & 0 & 0	\end{array}\right]\right\}.\end{aligned}$$
 						Then $X_6=T_6X'_6T_6$ and 
 						$$
 							\begin{aligned}
 							X_6\in &\left\{\mat{X^{A_1}_2}{0}{0}{U_2}, \ \left[\begin{array}{cccccc}
 							1/2 & 0 & 1/2 & 0 & 0 & 0 \\
 							0 & 0 & 0 & 0 & 1/2 & 1/2 \\
 							1/2 & 0 & 1/2 & 0 & 0 & 0 \\
 							0 & 0 & 0 & 0 & 1/2 & 1/2 \\
 							0 & 1/2 & 0 & 1/2 & 0 & 0 \\
 							0 & 1/2 & 0 & 1/2 & 0 & 0
 						\end{array}\right], \right.\\ &\left. \left[\begin{array}{cccccc}
 						0 & 0 & 0 & 0 & 1/2 & 1/2 \\
 						0 & 1/2 & 0 & 1/2 & 0 & 0 \\
 						0 & 0 & 0 & 0 & 1/2 & 1/2 \\
 						0 & 1/2 & 0 & 1/2 & 0 & 0 \\
 						1/2 & 0 & 1/2 & 0 & 0 & 0 \\
 						1/2 & 0 & 1/2 & 0 & 0 & 0
 						\end{array}\right]\right\},\end{aligned}$$
 					where $X^{A_1}_2$ is determined in \eqref{XA12}. Notice that only the first matrix in the latter is contained in the solution set $\overline{X_{12}}$ while the remaining two members are not. We define 
 					\begin{eqnarray}
 						\label{X6}
 						\begin{aligned} &\overline{X_6}=\\
 							&\left\{ \left[\begin{array}{cccccc}
 								1/2 & 0 & 1/2 & 0 & 0 & 0 \\
 								0 & 0 & 0 & 0 & 1/2 & 1/2 \\
 								1/2 & 0 & 1/2 & 0 & 0 & 0 \\
 								0 & 0 & 0 & 0 & 1/2 & 1/2 \\
 								0 & 1/2 & 0 & 1/2 & 0 & 0 \\
 								0 & 1/2 & 0 & 1/2 & 0 & 0
 							\end{array}\right], \left[\begin{array}{cccccc}
 								0 & 0 & 0 & 0 & 1/2 & 1/2 \\
 								0 & 1/2 & 0 & 1/2 & 0 & 0 \\
 								0 & 0 & 0 & 0 & 1/2 & 1/2 \\
 								0 & 1/2 & 0 & 1/2 & 0 & 0 \\
 								1/2 & 0 & 1/2 & 0 & 0 & 0 \\
 								1/2 & 0 & 1/2 & 0 & 0 & 0
 							\end{array}\right]\right\}.
 					\end{aligned}\end{eqnarray}\item Similarly, when $W_1=\operatorname{span}\{(1,1,1,1,1,1)^t\}^\perp$, then the only solution is $X_0=U_6$.
 						
 						\item Finally, if $W_0=W^{A_1}_0\oplus Q^{A_2}_0$ then we obtain all permutation solutions. By Example \ref{Ex_4,2}, we know that all permutation solutions are 	
 						\begin{eqnarray}\label{perms6}\begin{aligned}
 							\S_P=	\{&
 								(1\,2\,3\,4)(5\,6),\,
 								(1\,4\,2\,3)(5\,6),\,
 								(1\,2\,4\,3)(5\,6),\,
 								(1\,3\,2\,4)(5\,6),\\
 								&(1\,3\,4\,2)(5\,6),
 								(1\,3)(2\,6\,4\,5),\,
 								(1\,3)(2\,5\,4\,6),\,
 								(1\,6\,3\,5)(2\,4),\,
 								(1\,5\,3\,6)(2\,4)
 								\}.
 							\end{aligned}\end{eqnarray}
 						
 	\end{itemize}
		
	In summary, we have $9$ permutation solutions, given in \eqref{perms6}, and proper doubly stochastic solutions given as the union of the sets
	$$\S_{pDS}=\{U_6\}\cup \overline{X_{12}}\cup \overline{X_6}.$$	
		
		\noindent These are all doubly stochastic solutions. \hfill$\clubsuit$
	\end{example}
	
		\subsection*{Concluding remarks}
	In this paper we have demonstrated how to, at least theoretically, find all doubly stochastic solutions to $AXA=XAX$ whenever $A$ is an arbitrary permutation input matrix. We have found another way to chacterize the set of permutation solutions, and have shown how finding d.s. solutions essentially becomes finding suitable permutation solutions. We have also shown that every d.s. solution is a normal partial isometry, which implies that every d.s. solution has an inner inverse which is also doubly stochastic. Prior to our knowledge, no such characterization has been done regarding this particular case of the Yang-Baxter-like matrix equation. By publishing these results, we offer a correction and completion of some previous works regarding this problem. \\
	
	\noindent\textbf{Comment.} These results were obtained as a part of the Student Internship 2025/2026 organized by the Mathematical Institute of the Serbian Academy of Sciences and Arts.\\
	
	\noindent\textbf{Conflict of interest.} The authors declare that there are no conflicts of interest in publishing the findings obtained in this paper.\\
	
	\noindent\textbf{Data availability statement.} Data availability is not applicable as no data sets were generated during this research.\\
	
	\noindent\textbf{Funding.} The first author is supported by the Ministry of Education, Science and Technological Development, Republic of Serbia, grant No. 451-03-33/2026-03/200029. The second author is supported by the Ministry of Education, Science and Technological Development, Republic of Serbia, grant No. 451-03-34/2026-03/200124.

\end{document}